\documentclass[oneside,a4paper,11pt]{amsart}	
\usepackage[francais,english]{babel}
\usepackage{mathrsfs}
\usepackage{amsmath}
\usepackage{amsthm}
\usepackage{amsfonts}
\usepackage{amssymb}
\usepackage{amscd}
\usepackage{setspace}
\usepackage[left=3cm,right=3cm,top=3cm,bottom=3cm]{geometry}						
\usepackage{graphicx}
\usepackage[toc,page]{appendix}
\usepackage{lipsum}
\usepackage{url}

\theoremstyle{plain}							
\newtheorem{thm}{\textbf{Theorem}}[section]
\newtheorem{prop}[thm]{Proposition}
\newtheorem{cor}[thm]{Corollary}

\newtheorem{lem}[thm]{Lemma}

\theoremstyle{definition}						

\newtheorem{defi}[thm]{Definition}
\newtheorem{rem}[thm]{Remark}
\newtheorem{ex}[thm]{Example}

\theoremstyle{remark}							

\newcommand{\Iff}{\Leftrightarrow}				
\renewcommand{\and}{\wedge}						

\newcommand{\onto}{\twoheadrightarrow}			
\newcommand{\into}{\hookrightarrow}				

\newcommand{\NN}{\mathbb N}						
\newcommand{\ZZ}{\mathbb Z}						
\newcommand{\QQ}{\mathbb Q}						
\newcommand{\RR}{\mathbb R}						

\newcommand{\proj}{\varprojlim}					

\newcommand{\xx}{\times}

\renewcommand{\phi}{\varphi}					

\newcommand{\ms}{\mathscr}						
\newcommand{\mf}{\mathfrak}						
\newcommand{\ov}{\overline}						
\newcommand{\m}[1]{\mathsf{#1}} 				
\newcommand{\wt}{\widetilde}					

\newcommand{\subs}{\subseteq}
\newcommand{\sups}{\supseteq}

\newcommand{\srad}{\sqrt[s]}

\newcommand{\p}{\mf p}							

\renewcommand{\m}{\mf m}							
\newcommand{\n}{\mf n}

\renewcommand{\O}{\ms O}								

\newcommand{\<}{\langle}
\renewcommand{\>}{\rangle}

\newcommand{\erad}{\sqrt[e]}

\title{Equiresidual Algebraic Geometry I: The Affine Theory}
\author{Jean Barbet}
\thanks{We would like to thank the IRMA of Strasbourg (France) for granting us access to their scientific library}

\begin{document}

\maketitle


\begin{abstract}
In this first work dedicated to the generalisation of classic algebraic geometry to non algebraically closed fields and axiomatisable classes of fields, we develop the foundations for equiresidual algebraic geometry (EQAG), i.e. algebraic geometry over any commutative field $k$, algebraically closed or not. 
This is possible thanks to the existence in non algebraically closed fields of many normic forms, i.e.  homogeneous polynomials with no non-trivial zero rational in $k$, and relies on an equiresidual generalisation of Hilbert's Nullstellensatz and the Jacobson radical in finitely presented $k$-algebras. Usual algebraic constructions are naturally worked out using a new type of $k$-algebras which correspond to localisations of $k$-algebras by means of polynomials over $k$ with no inner zero. 
The theory leads to a fruitful characterisation of the sections of sheaves of regular functions over affine algebraic sets, providing a dualisation of the (equiresidual) affine algebraic varieties over $k$ by an equiresidual analogue of reduced algebras of finite type. 
The true generalisation of reduced algebras are defined using the normic forms and encompass all rings of \emph{rational} functions over affine subvarieties. These \og special" algebras are interpreted as giving rise to a \og special" radical, an equiresidual version of the classic radical of an ideal, of which we give a clean algebraic characterisation. Special extensions are the \og right" field extensions of the ground field in EQAG and provide a connexion with model-theoretic algebraic geometry, as exemplified with real closed fields and p-adically closed fields. The special radical is well-behaved with respect to localisation, and this leads to an equiresidual analogue of the prime spectrum, which connects affine EQAG with scheme theory.
\end{abstract}

\tableofcontents

\section{Introduction}
This is the first part of a series or works dedicated to the generalisation of classic algebraic geometry to non algebraically closed fields and axiomatisable classes of fields, both subjects being intimately and fundamentally connected. Indeed, if we consider algebraic geometry over an algebraically closed field as a classic paradigm, we may go in either one of two directions : change the ground field for any field, or change the category of algebraically closed fields for another (axiomatisable) category. In this first part we develop the foundations for \emph{equiresidual algebraic geometry} (EQAG).

\subsection*{The problem and aim of equiresidual algebraic geometry}
EQAG aims at extending classic algebraic geometry to any commutative field, i.e. algebraically closed or not. This is a most natural problem, since classic methods in linear algebra and affine and projective geometry are available over any commutative field and the basic objects of algebraic geometry may be defined inside any such, and except if we are not aware of it, one only wonders why it has not been systematically undertaken before. The first step we take here is building affine algebraic geometry \emph{inside} or \emph{relative to} any field, for its own sake and in the perspective of an extension to categories of \og similar" fields, for instance satisfying the axioms of a first order theory : the numerous examples coming from number theory indeed show a close relationship between arithmetic, differential, and definable aspects of their algebraic geometry, and EQAG turned out to be the starting ground for the sound systematic approach of \cite{PAG}.
Furthermore, the now classic and geometrically essential theory of henselian local rings and henselisation is truly written in an \og equiresidual" context, in the sense that geometrically speaking, one considers rational points of smooth varieties in the residual field of a local ring, rather than in a separable algebraic closure of this field, as opposed to strict henselisation (\cite{ALH}, VIII).\\
This problem and the present solution should not be confused with the usual method which consists in using universal domains (\cite{FA}, Chapter 10) or scheme theory over any commutative field $k$ (\cite{HAG}, Chapter 2) for the purpose of doing algebraic geometry over $k$ in the sense of \emph{with parameters in $k$}. Indeed, such an approach considers all rational points of varieties defined over $k$, virtually inside every (finite, algebraic) extension of the field $k$, and it paradoxically \emph{excludes much structure} on certain non-algebraically closed fields. For instance, the principle of \emph{real algebraic geometry} is to enrich the usual structure of algebraic varieties inside a real closed field $R$ by using polynomials \emph{which have no zero rational in such fields}, as for instance the polynomials of the form $P(X)=1+X_1^2+\ldots+X_n^2$, thanks to which it is possible to define an order on $R$, as well as semi-algebraic subsets of the afine spaces $R^n$. If we were to follow the approach of scheme theory over $R$, we would add all zeros of polynomials of the form $P$ and nullify the very purpose of real algebraic geometry ! Furthermore, by the Tarski-Seidenberg principle (\cite{RAG}, Theorem 1.4.2), essentially every semi-algebraic subset of $R^n$ which has a point rational in a real extension of $R$ already has one in $R$, so the basic semi-algebraic structure of real closed fields is essentially built on (sets of) polynomials which \emph{already} have zeros rational inside such a field. This shows that our problem is essentially relevant to peculiar features of algebraic geometry over \emph{non-algebraically closed fields}, which cannot be adressed by scheme theory alone or any approach consisting solely of looking at as many rational points of varieties as possible : we have to \emph{exclude} some of these points. The \og equiresidual" point of view in real algebraic geometry also clearly appears in the consideration of Nash functions and the characterisation of their germs using the equiresidual procedure of henselisation (\cite{RAG}, Chapter 8). All this shows that the common opinion that by working inside any commutative field, we \og loose some structure" at some point, is misguided.\\
Of course, the example of real algebraic geometry is one among many notable ones in model theory : in certain very specific cases, it is possible to develop whole analogues of complex and real algebraic geometry by using some specific features of a field or a family of fields often identified by a set of (first order) axioms. In particular, one knows in some core examples how to interpret the model-theoretic notion of quantifier elimination as some analogue of Chevalley's theorem on constructible sets (see Proposition 5.2.2 in \cite{RAG} for real algebraic geometry, and \cite{BEL2} for a $p$-adic analogue). In other words, if a given field $k$ is a model of what is known as a certain \emph{model-complete theory} of fields, it is often in the natural cases possible to develop the corresponding algebraic geometry in a (first order) axiomatisable class of fields : real (closed) fields for real algebraic geometry, $p$-adically closed fields for $p$-adic algebraic geometry, and so on... However, this approach is limited in several ways : first, for each particular case, one has to develop the corresponding theory, and there is no universal framework at this time; secondly, we have to know that our favorite field is the model of a model-complete theory, and often to have an explicit axiomatisation of such a theory, in order to conveniently work with it, which is not always the case (for instance, for $\QQ$ itself !); thirdly, we have to use first order logic at the outset, and it seems at least odd that a universal framework for algebraic geometry over any field should be founded on logic. In the present work we adress these three limitations, while providing a sound and purely algebraic foundation for algebraic geometry over any commutative field, retaining its particular features; as a development we will adress the problem of building a universal model-theoretic framework in a forthcoming publication.

\subsection*{Contents}
The article is structured in two parts, each containing two sections. In Part I, we deal with the \og strictly" equiresidual theory, i.e. all that pertains to the basic structures surrounding affine algebraic geometry inside a ground field $k$. In Part II, we deal with the proper algebraic objects with which to develop EQAG and its extension to model-theoretic algebraic geometry.\\
In Section \ref{EQNUL}, we found the affine approach to EQAG by investigating the basic structures surrounding affine \emph{subvarieties}. A certain \emph{equiresidual} Nullstellensatz (Theorem \ref{NSREL}) holds in every commutative field and rests on the existence in all non-algebraically closed fields of \emph{normic forms}, i.e. homogeneous polynomials having only the trivial zero. Characterising the so-called \emph{essential} maximal ideals of finitely generated algebras over a field $k$ which have points rational in $k$, we carve an analogue of the classic radical of an ideal, the \emph{equiresidual radical}, closely related to an \emph{algebraic signature} of the field, clarifying the \og dictionary" between affine algebraic subvarieties and ideals, and specifying the affine duality with the now \emph{essential} algebras of finite type. The \emph{canonical localisation} of a $k$-algebra leads to the abstract notion of \emph{$*$-algebra} - providing the \og right" category of algebras in which to work out EQAG in general - and to the elementary characterisation of the equiradical (Theorem \ref{CARERAD}), which sheds some light on essential $*$-algebras of finite type as such.\\
In Section \ref{SEQVAR}, we expand the strictly equiresidual affine theory by developing the \emph{intrinsic} point of view. We characterise the local sections of sheaves of regular functions on affine subvarieties (Theorem \ref{MAG.3.6.a}) using the equiradical and canonical localisation. The rings of global sections differ in general from the coordinate rings, therefore the classic abstract duality between affine varieties and reduced algebras of finite type has to be reconsidered. The algebraic pole is now played by essential $*$-algebras of finite $*$-type, whereas we introduce the geometric pole, \emph{affine algebraic equivarieties}, as a full subcategory of a natural category of \emph{equiresidual varieties}, locally ringed spaces over the base field $k$ among which we will also subsequently find projective spaces and varieties. The dualisation uses a natural concrete version of the maximal spectrum functor which fits in.\\
In Section \ref{SPECALGRAD}, we extend the strictly equiresidual algebraic theory to the natural setting in which we find the algebras surrounding EQAG. We generalise essential algebras to \emph{special algebras}, naturally associated to the \emph{homogeneous signature} of the ground field $k$, which may be characterised as the set of all its normic forms. We introduce the associated notion of \emph{special ideal} as the true intended generalisation of the radical and prime ideals of the classic theory, allowing us to consider rational function fields as \emph{special extensions} of $k$. These turn out to be the extensions preserving the algebraic signature of $k$, thus handling some implication of EQAG in model-theoretic algebraic geometry, as exemplified in the real and $p$-adic cases.\\
In Section \ref{LOCFUN}, we investigate the localisation of special algebras and apply it to other kinds of functions naturally associated with EQAG. We exchange the equiradical for the \emph{special radical}, of which an elementary characterisation using algebraic and homogeneous signatures is given in Theorem \ref{CARSPERAD}. Special ideals happily turn out to be exactly those equal to their special radical, and this entails that the special $*$-algebras of finite $*$-type are exactly the essential ones, completing the affine theory. The special radical commutes with any kind of localisation, and this allows us to describe rational functions over any affine variety as special $*$-algebras. These in general appear to have an exact representation as global sections of the special prime spectrum, an analogue of the prime spectrum which fits affine algebraic equivarieties into the framework of scheme theory.

\subsection*{Preliminaries and conventions}
All rings and fields considered are implicitly unitary and commutative and we use some standard notation, terminology and folklore from commutative algebra and algebraic geometry, which we briefly review and complete. If $k$ is a field and $I$ is an ideal of a polynomial algebra $k[X_1,\ldots,X_n]$, the corresponding (affine) algebraic subset of $k^n$ is noted $\ms Z(I)=\{P\in k^n : \forall f\in I, f(P)=0\}$. These algebraic subsets of $k^n$ are the closed sets of a Noetherian topology called the \emph{Zariski topology}; recall that in general any nonempty open subset of an irreducible closed set is dense and irreducible (\cite{HAG}, Example 1.1.3). If $V\subs k^n$ is an algebraic subset, the \emph{coordinate ring (or algebra)} of $V$ is the $k$-algebra $k[V]:=k[X_1,\ldots,X_n]/\ms I(V)$, where $\ms I(V)=\{f\in k[\ov X] :\forall P\in V,\ f(P)=0\}$ is the \emph{ideal of $V$}; any element $f\in k[V]$ defines a \emph{function} $V\to k$, $P\mapsto F(P)$, for any $F\in k[\ov X]$ such that $F+\ms I(V)=f$. By definition of the induced Zariski topology on $V$, any basic open subset will be denoted by $D_V(f)=\{P\in V : f(P)\neq 0\}=V-\ms Z_V(f)$ for a certain $f\in k[V]$.\\
It is mentioned in \cite{CLO}, Exercise 4.1.7, that if $k$ is any non-algebraically closed field, then for every $n>1$ there exists a non-constant homogeneous polynomial $P\in k[X_1,\ldots,X_n]$ such that the only zero of $P$ in $k^n$ is $(0,\ldots,0)$ (choose any non-zero polynomial $P\in k[X]$ in one variable with no zero rational in $k$, its homogenisation $P^\#(X,Y)$ has this property, and use induction on $n$). This justifies in general the following definition after McKenna (\cite{MK}, Lemma 4) :

\begin{defi}\label{NORMFORM}
We say that a non-constant homogeneous polynomial $P$ over $k$ is a \emph{normic form} if it has only the trivial zero rational in $k$.
\end{defi}
\begin{rem}\label{REMNORM}
This is for the moment only an \emph{analogue} of McKenna's model-theoretic definition. We develop a common generalisation using positive logic in \cite{PAG}.
\end{rem}

\noindent The only constant normic form over any field $k$ is $0$, and the normic forms in one variable are the non-zero monomials. Normic forms are useful - at least in \emph{non}-algebraically closed fields - in order to reduce the description of algebraic sets to sets of zeros of a unique polynomial, and we will make extensive use of this fact.\\
If $U\subs V$ is an open subset (for the induced Zariski topology on $V$), a function $f:U\to k$ is called \emph{regular at $P\in V$} if there exists an open neighbourhood $U_P\subs U$ of $P$ in $U$ and elements $g,h\in k[V]$ such that for all $Q\in U_P$, $h(Q)\neq 0$ and $f(Q)=g(Q)/h(Q)$; $f$ is called \emph{regular (over $U$)} if $f$ is regular at every $P\in U$ (and $f$ is then continuous). The set of regular functions over $U$ is written $\O_V(U)$ and $\O_V$ is a sheaf of $k$-algebras, called the \emph{sheaf of regular functions on $V$}. An \emph{affine (algebraic) subvariety of $k^n$} is a pair $(V,\O_V)$, where $V\subs k^n$ is an algebraic subset and $\O_V$ is its sheaf of regular functions. An element of the stalk $\O_{V,P}$ of $\O_V$ at $P\in V$ will be noted $[g,U]$, where $P\in U\subs V$ and $g\in \O_V(U)$, or simply $[g]$. The following proposition - which we will refer to as the \og stalk lemma" - should be folkloric but we have never read it elsewhere (in usual textbooks on algebraic geometry, it is proved on algebraically closed fields as a consequence of Hilbert's Nullstellensatz, see \cite{HAG}, Theorem I.3.2 for instance !) :

\begin{prop}[Stalk Lemma]
\label{SMALEM}
For any affine algebraic subvariety $V\subs k^n$ and for any $P\in V$, we have $\O_{V,P}\cong k[V]_{\m_P}$ for $\m_P=\{f\in k[V] : f(P)=0\}$. In particular, the structural morphism $k\to \ov{\O_{V,P}}$ is an isomorphism onto the residual field of $\O_{V,P}$.
\end{prop}
\begin{proof}
Write $A=k[V]$ and $\m=\m_P$. Let $[f]\in \O_P=\O_{V,P}$ : there is a neighbourhood $U$ of $P$ in $V$ and $a,g\in A$ with $g(Q)\neq 0$ and $f(Q)=a(Q)/g(Q)$ for every $Q\in U$, whence $U\subs D_V(g)$ and we may assume that $U=D_V(g)$ with the same data. As $g\in A-\m$, define $\phi([f]):=a/g\in A_\m$ : if $[f]=[b/h]$ in $\O_P$, there exists a basic open neighbourhood $U'=D_V(l)\subs U$ of $P$ in $V$ on which $a/g\equiv b/h$; we have $D_V(l)\subs D_V(g)\cap D_V(h)=D_V(gh)$ and $l\in A -\m$, and the regular map defined by $(ah-bg)/gh$ on $D_V(l)=D_V(ghl)$ is zero, hence $ahl-bgl$ is also zero on $D_V(l)$, whereas for $Q\in V-D_V(l)$, we have $ghl(Q)=0$, and therefore $(ahl-bgl)(ghl)$ is zero on $V$, and hence in $A$. As $ghl\notin \m$, we get $ahl-bgl=0$ in $A_\m$, whereby $a/g=b/h$ in $A_\m$ and $\phi$ is well defined, and obviously a $k$-morphism. Finally, if $a/g\in A_\m$, the regular map defined by $a/g$ on $D_V(g)$ has $\phi([a/g,D_V(g)])=a/g$, and if $\phi([f])=0$ with $f$ defined as before on $D_V(g)$ by $a/g$ say, as $a/g=0$ in $A_\m$ there is $h\in A-\m$ with $ha=0$ in $A$, whence $[a/g]=[ah/gh|_{D_V(gh)}]=0$, and $\phi$ is an isomorphism. The isomorphism $\O_{V,P}\cong k[V]_{\m_P}$ induces a residual isomorphism $\ov{\O_{V,P}}\cong k[V]_{\m_P}/\m_P k[V]_{\m_P}$ over $k$, and by definition of $\m_P$ the structural morphism $k\to \ov{\O_{V,P}}$ is an isomorphism. 
\end{proof}

\noindent If $W\subs k^m$ is another algebraic subset, a \emph{regular morphism} from $V$ to $W$ is a map $f:V\to W$ such that there exist $f_1,\ldots,f_m\in k[V]$ with $f(P)=(f_1(P),\ldots,f_m(P))$ for all $P\in V$. Any regular morphism $f=(f_1,\ldots,f_m):V\to W$ induces in turn a $k$-algebra morphism $k[f]:k[W]\to k[V]$ in the usual way : to every $g=G+\ms I(W)\in k[W]$ we associate $G(F_1,\ldots,F_m)+\ms I(V)$, if $f_i=F_i+\ms I(V)$ for $i=1,\ldots,m$, and this defines a full and faithful functor $k[-]$ from the dual category of affine algebraic sets and regular morphisms into the category of reduced $k$-algebras of finite type. If $f:V\to W$ is a regular morphism of affine algebraic subvarieties, for every open subset $U\subs W$ and for every $s\in \O_W(U)$, we have $f\circ s\in \O_V(f^{-1}U)$, the map $f^\#_U:s\in\O_W(U)\mapsto f\circ s\in \O_V(f^{-1}U)$ is a morphism of $k$-algebras and the $f^\#_U$'s define a sheaf morphism $f^\#:\O_W\to f_*\O_V$. For every $P\in V$ we have a residual $k$-morphism $f^\#_P:\O_{W,f(P)}\to \O_{V,P}$ induced by the universal property of stalks considered as inductive limits, which is local by Lemma \ref{SMALEM}, so $(f,f^\#):(W,\O_W)\to (V,\O_V)$ is a morphism of locally ringed spaces in $k$-algebras, and we have a functor $f\mapsto (f, f^\#)$ from the dual category of affine algebraic subvarieties to locally ringed spaces in $k$-algebras. If $A$ is any ring, we note $Spm(A)$ the maximal spectrum of $A$, i.e. the set of all maximal ideals of $A$, implicitly topologised as usual by taking as basic open sets the subsets of the form $D(f)=\{\m\in Spm(A) : f\notin \m\}$; this is the \emph{Zariski topology on $A$}.

\part{Affine Equiresidual Algebraic Geometry}
\section{The equiresidual Nullstellensatz and its affine consequences}\label{EQNUL}
Hilbert's Nullstellensatz is the fundamental theorem of algebraic geometry over algebraically closed fields, or with points rational in such. For EQAG we have to generalise it over any commutative field $k$, which we tackle in this section with Theorem \ref{NSREL}. The key ingredient is to exclude the polynomials with no zero rational in $k$, and somewhat virtually work in $k$-algebras where they are invertible. We introduce the natural \og equiradical" of an ideal in finitely generated $k$-algebras, which allows us to extend the duality between affine subvarieties and what we call here essential ideals. Both considerations lead us to consider two kinds of algebras over $k$, $*$-algebras and essential algebras.

\subsection{The \"Aquinullstellensatz and the equiradical}\label{ERAD}
If $k$ is any field, by excluding the ideals of finitely generated $k$-algebras which have no zero rational in $k$ and using normic forms (see the Introduction) whenever $k$ is not algebraically closed, we may generalise Hilbert's Nullstellensatz and \cite{CLO}, Exercise 4.1.8 as the following \og \"Aquinullstellensatz", an equiresidual generalisation valid over any commutative field.

\begin{thm}[\"Aquinullstellensatz]\label{NSREL}
Let $k$ be any field, $A$ a finitely generated $k$-algebra, and $S$ the set of all $f\in A$ such that $\phi(f)\neq 0$ for all $k$-morphisms $\phi:A\to k$. Every ideal $I$ of $A$ disjoint from $S$ and maximal as such is a maximal ideal such that $A/I\cong k$ (and reciprocally).
\end{thm}
\begin{proof}
If $k$ is algebraically closed, then $S=k^\xx$ and the result is a consequence of Hilbert's Nullstellensatz, so we now suppose that $k$ is not algebraically closed. If $P\in k[\ov X]=k[X_1,\ldots,X_n]$ has a zero $[\ov f]=\ov f+I$ in $A/I$ for $\ov f\in A^n$, we have $P(\ov f)\in I$, and as $I\cap S=\emptyset$, there exists a $k$-morphism $\phi:A\to k$ such that $P(\phi(\ov f))=\phi(P(\ov f))=0$, and $P$ already has a zero in $k$. In particular, if $I=(P_1,\ldots,P_m)$ and $N(X_1,\ldots,X_m)$ is a normic form for $k$ (Definition \ref{NORMFORM}), as $N(P_1,\ldots,P_m)$ has a zero in $A/I$, it has a zero in $k$ by what precedes, and as $N$ is a normic form, $I$ itself has a zero in $k$, corresponding by evaluation to a $k$-morphism $e:A/I\to k$. Now the composite $k$-morphism $\phi :A\to A/I\to k$ has $I\subs Ker(\phi)$, and if $P\in Ker(\phi)$, by definition we have $P\notin S$, so $Ker(\phi)\cap S=\emptyset$ : by maximality of $I$ with this property, we have $I=Ker(\phi)$, and thus $e:A/I\to k$ is an isomorphism, and $I$ is maximal.
\end{proof}
\begin{rem}
i) Finiteness is needed in both cases, in the first for the application of Hilbert's Nullstellensatz, in the second for the application of a normic form to finitely many generators.\\
ii) For any $k$-algebra $A$ and ideal $I$ of $A$, if $\phi:A/I\cong k$ is an isomorphism, $\phi$ is necessarily the inverse of the structural morphism $k\to A/I$, so the ideals of the statement are exactly those for which $k\cong_k A/I$.
\end{rem}

\begin{cor}\label{CORNSREL}
If $k$ is any field, $V\subs k^n$ is an affine algebraic subvariety, and $S=\{g\in k[V] :\forall P\in V,g(P)\neq 0\}$, then an ideal $I$ of $k[V]$ has a zero in $V$ if and only if $I\cap S=\emptyset$.
\end{cor}
\begin{proof}
If $I$ has a zero in $V$, we have $I\cap S=\emptyset$. Conversely, if $I\cap S=\emptyset$, by Noetherianity of $k[V]$ there exists an ideal $\m$ of $k[V]$, containing $I$, disjoint from $S$, and maximal with this property : by Theorem \ref{NSREL}, the structural morphism $k\to k[V]/\m$ is an isomorphism, so that $I$ has a zero in $V$.
\end{proof}

\noindent As a first significant geometric consequence of the \"Aquinullstellensatz (ANS), we may characterise the global sections of the sheaf of regular functions on an \emph{irreducible} affine algebraic subvariety, a result we will reinterpret and generalise in Section \ref{AFFSUB}.

\begin{prop}\label{CHSEC}
If $V\subs k^n$ is irreducible and $k\{V\}:=k[V]_S$, where $S=\{g\in k[V]:\forall P\in V,g(P)\neq 0\}$, then $\Gamma(V,\O_V)\cong k\{V\}$.
\end{prop}
\begin{proof}
Let $f\in \Gamma(V,\O_V)$ and for every $P\in V$, $U_P\subs V$ an open neighbourhood of $P$ in $V$ such that $f|_{U_P}\equiv u_P/v_P$, with $u_P,v_P\in k[V]$. As $v_P\neq 0$ for all $P$, define a map $\phi:\Gamma(V,\O_V)\into k(V)$ by $f\mapsto u_P/v_P$ for any $P$; if $P,Q\in V$, as $V$ is irreducible $U_P$ is dense in $V$, so $O:=U_P\cap U_Q\neq \emptyset$ and for every $R\in O$ we have $f(R)=u_P(R)/v_P(R)=u_Q(R)/v_Q(R)$ and thus $u_Pv_Q|_O=u_Qv_P|_O$ : by density of $O\neq\emptyset$ in $V$, as the diagonal $\Delta_V$ is closed in $V\xx V$ we have $u_Pv_Q=u_Qv_P$ in $k[V]$, and thus $u_P/v_P=u_Q/v_Q$, so $\phi$ is well defined, and obviously a $k$-morphism. If $\phi(f)=u_P/v_P=0$, we have $u_P=0\in k(V)$, so $f|_{U_P}\equiv 0$ and as $f$ is continuous and $U_P$ is dense, as again $\Delta_V$ is closed we have $f\equiv 0$, so $\phi$ is injective : let $A$ be its isomorphic image in $k(V)$, by definition we have $k\{V\}\subs A$. Now let $I$ be the ideal of $k[V]$ generated by the $v_P$'s, $P\in V$ : if $I\cap S=\emptyset$, by the ANS (\ref{NSREL}) $I$ has a rational point $Q\in V$, for which $v_Q(Q)=0$, which is impossible, so there exist $v\in I\cap S$, $r\in\NN$, $P_1,\ldots,P_r\in V$ and $\alpha_1,\ldots,\alpha_r\in k[V]$ with $v=\sum_{i=1}^r \alpha_i v_{P_i}$, from which we get, in $A$, $\phi(f)v=\sum_i \alpha_i \phi(f) v_{P_i}=\sum_i \alpha_i u_{P_i}$, and therefore $\phi(f)=(1/v)\sum_i \alpha_i u_{P_i}\in k\{V\}$, and finally $k\{V\}=A$, i.e. $ \Gamma(V,\O_V)\cong k\{V\}$.
\end{proof}

\noindent If $k[\ov X]=k[X_1,\ldots,X_n]$ is a polynomial algebra and $I$ is an ideal of $k[\ov X]$, write $e_P:k[\ov X]\to k$ be the evaluation morphism at $P\in k^n$ : the elements of $\ms Z(I)$ are in bijection with the $k$-morphisms $e_P:k[\ov X]\to k$ such that $I\subs Ker(e_P)$ and with the same notations, by Theorem \ref{NSREL} the maximal ideals disjoint from $S$ and containing $I$ are thus the $Ker(e_P)$'s for $P\in \ms Z(I)$. It follows that $\ms I(\ms Z(I))$, the kernel of the product $k$-morphism $e_I:k[\ov X]\to k^{\ms Z(I)}$ of the $e_P$'s for $P\in \ms Z(I)$, is the intersection of all maximal ideals of $k[\ov X]$ containing $I$ and disjoint from $S$. Abstracting this notion we adopt the following 

\begin{defi}\label{EQUIRAD}
If $A$ is a $k$-algebra, say that a maximal ideal $\m$ of $A$ is \emph{essential} if the structural morphism $k\to A/\m$ is an isomorphism. If $I$ is any ideal of $A$, the \emph{equiresidual radical of $I$}, or \emph{equiradical of $I$}, noted $\erad I$, is the intersection of all essential maximal ideals of $A$ containing $I$. 
\end{defi}
\begin{rem}\label{EQUIREM}
i) If $A$ is \emph{of finite type} and $S=\{f\in A | \forall \phi:A\to k,\phi(f)\neq 0\}$ as before, then by Theorem \ref{NSREL} a maximal ideal $\m$ of $A$ is essential if and only if $\m\cap S=\emptyset$.\\
ii) If $A$ is the coordinate algebra $k[V]$ of some affine algebraic subvariety $V\subs k^n$, then the equiradical of an ideal $I$ of $A$ is nothing else than the ideal $\ms I(\ms Z_V(I))$ of $\ms Z_V(I)$.
\end{rem}

\noindent If $A$ is a $k$-algebra, the set $S$ as defined above is multiplicative; in case $A$ is \emph{of finite type}, by what precedes we may identify the essential maximal ideals of $A$ with the maximal ideals of $A_S$ by localisation. This leads to a transposition of the usual algebraic constructions surrounding classic algebraic geometry into this kind of localised algebras, which we begin to study here using a more convenient description, and allowing a profitable characterisation of the equiradical, thanks to the following notions which are inspired by Theorem 2 of \cite{MK} and Theorem 2.1 of \cite{B-H}.

\begin{defi}\label{ALGSIG}
i) The \emph{algebraic signature of $k$} is the set $\ms D$ of all polynomials in finitely many variables over $k$ which have no zero rational in $k$.\\
ii) If $A$ is a $k$-algebra, we note $M_A$ the multiplicative subset of all $D(\ov a)$ for $D\in \ms D$ and $\ov a\in A$, and we call $A_M:=A_{M_A}$ the \emph{canonical localisation of $A$}.
\end{defi}
\begin{rem}\label{REMCANLOC}
i) The algebraic signature is (only) an analogue of McKenna's \og determining sets" (\cite{MK}, Theorem 2). As with normic forms, both notions have a common natural generalisation using positive logic (see Remark \ref{REMNORM}).\\
ii) If $k$ is algebraically closed, then $\ms D=k^\xx$, so for every $k$-algebra $A$, we have $M_A\cong k^\xx$ and $A\cong A_M$.\\
iii) Of course, in general $A_M$ is different from $A$. For instance, if $k=\RR$ and $A=\RR[X]$ then $X^2+1$ is invertible in $A_M$ but not in $A$.
\end{rem}

\begin{lem}\label{CARSPE}
If $A$ is a finitely generated $k$-algebra and $J$ is an ideal of $A$, then $J\cap S=\emptyset\Iff J\cap M_A=\emptyset$.
\end{lem}
\begin{proof}
It suffices to prove it for $A=k[\ov X]=k[X_1,\ldots,X_n]/I$. If $f\in J\cap S$, write $f=F+I$ with $F\in k[\ov X]$, and let $P_i:i=1,\ldots,m$ be finitely many generators of $I$.  Suppose $k$ is algebraically closed, by definition of $S$ the ideal $(F,I)$ of $k[\ov X]$ has no zero in $k$; by Hilbert's Nullstellensatz we have $1\in \sqrt{(F,I)}$, so that $1=gf$ in $A$ with $g=G+I$, for some $G\in k[\ov X]$; it follows that $1\in J$, so $J=A$ and $J\cap M_A\neq\emptyset$. Suppose $k$ is not algebraically closed, and $N(Y,Z_1,\ldots,Z_m)$ is an appropriate normic form over $k$ : by definition of $S$, $F$ and the $P_i$'s have no common zero in $k$, so the polynomial $D=N(F,P_i:i)$ has no zero in $k$ and is therefore a member of $\ms D$. It follows that $g:=N(f,\ov 0)=N(F,P_i:i)+I=D(\ov X+I)$ is both a member of $J$ (as a $k$-linear combination of powers of $f$) and a member of $M_A$. Conversely, it suffices to show that $M_A\subs S$, so suppose $f\in M_A$ : we have $f=D(\ov g)$ for some $D\in \ms D$ and $\ov g=\ov G+I$; if $\phi:A\to k$ is a k-morphism, we have $\phi(f)=\phi(D(\ov g))=D(\phi(\ov g))\neq 0$ by definition of $\ms D$, so $f\in S$, and the lemma is proved. \end{proof}
\begin{rem}
Keeping in mind the first point of Remark \ref{EQUIREM}, we now see that the essential maximal ideals of a finitely generated $k$-algebra $A$ are the maximal ideals which are disjoint from $M_A$. Beware that this is not true in $k$-algebras in general (see Example \ref{CORFUNEX}), which we will tackle in Section \ref{SPECALGRAD}.
\end{rem}

\begin{prop}
For every finitely generated $k$-algebra $A$, we have $A_M\cong A_S$.
\end{prop}
\begin{proof}
As $M_A\subs S$, it suffices to show by the universal properties of $A_M$ and $A_S$ that every member of $S$ becomes invertible in $A_M$. We have $A_M^\xx=\bigcap\{\m^c : \m\in Spm(A_M)\}$, and $\m\in Spm(A_M)\Iff \m=\n A_M$ for $\n$ disjoint from $M_A$ and maximal as such $\Iff \m=\n A_M$ for $\n$ disjoint from $S$ and maximal as such (by Lemma \ref{CARSPE}) $\Iff \m=\n A_M$ for $\n$ maximal and essential by Theorem \ref{NSREL}. Now let $s\in S$ : $s$ belongs to no essential maximal ideal of $A$, so by what precedes $s$ is invertible in $A_M^\xx$ and the proposition is proved.
\qed \end{proof}
\begin{rem}
A direct proof in the non-algebraically closed case along Lemma \ref{CARSPE} is instructive : if the members of $M_A$ are invertible, an element of $S$ has the form $F+I$ with $(F,I)$ having no zero in $k$, so $N(f,\ov 0)$ is in $M_A$, so is invertible; now $N(Y,\ov 0)$ is precisely the monomial where only the variable $Y$ corresponding to $f$ occurs, with a power $\geq 1$, so that inverting $N(f,\ov 0)$ entails inverting $f$.
\end{rem}

\noindent The following lemma is a generalisation of the existence of \og rational points" (i.e. morphisms to the base field) for any non-trivial finitely generated algebra over an algebraically closed field.

\begin{lem}\label{CARPRE}
If $A$ is a finitely generated $k$-algebra, then $A_M\neq 0$ if and only if there exists a $k$-morphism $A_M\to k$.
\end{lem}
\begin{proof}
It suffices to prove the direct sense. If $A_M\neq 0$, let $\m$ be a maximal ideal of $A_M$, and let $\p:=\m\cap A$, an ideal of $A$ disjoint from $M_A$, and maximal as such. By Lemma \ref{CARSPE}, $\p$ is disjoint from $S$, and maximal as such, so by the \"Aquinullstellensatz \ref{NSREL} $\p$ is maximal and $A/\p\cong k$. As $A_M/\m\cong (A/\p)_M\cong k$ as $k$-algebras, $\m$ is the kernel of a morphism $A_M\to k$. \end{proof}

\noindent In order to characterise the equiradical in finitely generated $k$-algebras $A$, we are going to use canonical localisation at one element. If $a\in A$, we let $\Sigma_a$ be the multiplicative subset generated by all elements of the form $a^mD^\#(\ov b,a^n)$, for $D(\ov X)=\sum_{\ov i} a_{\ov i} \ov X^{\ov i}\in \ms D$ of degree $d$ say, and $D^\#(\ov X,Y)=\sum_{\ov i} a_{\ov i} \ov X^{\ov i} Y^{d-|\ov i|}=Y^dD[\ov X/Y]$ its homogenisation (with $|\ov i|=i_1+\ldots+i_n$), $m,n\in\NN$ and $\ov b$ an appropriate tuple from $A$. We note $A_{\<a\>}$ the localisation $\Sigma_a^{-1}A$.

\begin{lem}\label{CARUNLOC}
If $A$ is a $k$-algebra and $a\in A$, then the map $$c/a^mD^\#(\ov b,a^n)\mapsto (c/a^m)/a^{nd} D(\ov b/a^n)$$ is a $k$-isomorphism $A_{\<a\>}\cong (A_a)_M$.
\end{lem}
\begin{proof}
Let $l_a:A\to A_a$ be the localisation at $a$, $l_M:A_a\to (A_a)_M$ the canonical localisation, $f_a:A\to A_{\<a\>}$ the localisation by $\Sigma_a$ and $\phi_a:A_a\to A_{\<a\>}$ the unique morphism such that $\phi_a\circ l_a=f_a$. For $D(\ov x)\in \ms D$ and $\ov b/a^m\in A_a$, we have $D(\ov b/a^m)=(1/a^{md})D^\#(\ov b,a^m)\in A_a$, which becomes invertible in $A_{\<a\>}$. By the universal property of $l_M$ (as an $A$-morphism), there exists a unique $A$-morphism $\phi:(A_a)_M\to A_{\<a\>}$ such that $\phi\circ l_M=\phi_a$, and by the universal property of $l_a$ this is the unique such that $\phi\circ l_M\circ l_a=f_a$. The other way round, any non-negative power $a^m$ of $a$ is invertible in $(A_a)_M$ and in $A_a$ we have $D^\#(\ov b,a^m)=a^{md}D(\ov b/a^m)$, which also becomes invertible in $(A_a)_M$. By the universal property of $f_a$, there exists a unique $\psi:A_{\<a\>}\to (A_a)_M$ such that $\psi\circ f_a=l_M\circ l_a$ and we get $\psi\phi l_Ml_a=\psi f_a=l_Ml_a$, whence $\psi\phi=1$ by universal property of localisations; likewise, we have $\phi\psi f_a=f_a$, whence $\phi\psi=1$, so that $\phi$ and $\psi$ are reciprocal isomorphisms. Now by definition, we have $\phi((c/a^m)/D(\ov b/a^n))=ca^{nd}/a^mD^\#(\ov b,a^n)$ and $\psi(c/a^mD^\#(\ov b,a^n))=(c/a^m)/a^{nd}D(\ov b/a^n)$. The maps are represented on the following diagram :
$$\begin{CD} A @>l_a>> A_a @>l_M>> (A_a)_M @= (A_a)_M\\
@Vf_aVV @V\phi_a VV @V\phi VV @AA\psi A\\
A_{\<a\>} @= A_{\<a\>} @= A_{\<a\>} @= A_{\<a\>}. \end{CD}$$
\end{proof}

\noindent The following theorem is the key ingredient to the characterisation of the rings of sections of regular functions over an open subset of an affine algebraic subvariety (Theorem \ref{MAG.3.6.a}).

\begin{thm}\label{CARERAD}
For a finitely generated $k$-algebra $A$ and an ideal $I$ of $A$, we have 
$\erad I=\{a\in A : I\cap \Sigma_a\neq\emptyset\}$. 
\end{thm}
\begin{proof}
Suppose $a\notin \erad I$ : by definition there exists an essential maximal ideal $\m$ of $A$ containing $I$ and such that $a\notin \m$. For all $m,n\in\NN$, we have $a^m,a^n\notin\m$ and as $A/\m\cong k$, for all $D\in \ms D$ with degree $d$ say and appropriate $\ov b\in A$ we have $a^mD^\#(\ov b,a^n)\notin \m$ (otherwise $D([\ov b]/[\ov a^n])=(1/[a^{nd}])D^\#([\ov b],[\ov a^n])=0$ in $A/\m$, contradicting the choice of $D$ and $\m$), so $I\cap \Sigma_a=\emptyset$ by primality of $\m$. Conversely, if $I\cap \Sigma_a=\emptyset$, then $((A/I)_{a+I})_M\cong (A/I)_{\<a+I\>}$ (by Lemma \ref{CARUNLOC}) $\cong\Sigma_a^{-1}(A/I)\cong A_{\<a\>}/\Sigma_a^{-1} I\neq 0$, so there exists a $k$-morphism $(A/I)_{\<a+I\>}\to k$ by Lemma \ref{CARPRE}, since $(A/I)_{a+I}$ is finitely generated. Let then $\m$ be the kernel of the composite morphism $A\to A/I\to (A/I)_{\<a+I\>}\to k$ : we have $a\notin \m$, and as $\m$ is essential we get $a\notin \erad I$. 
\end{proof}

\subsection{Points and subvarieties in affine spaces}
In classic algebraic geometry (i.e. over an algebraically closed field $k$), we have a well known correspondance between algebraic subsets of $k^n$ and radical ideals of $k[X_1,\ldots,X_n]$ (\cite{HAG}, Corollary I.1.4). This is true in general if we replace radical ideals by the following notion :

\begin{defi}\label{ESSIDEAL}
If $A$ is a $k$-algebra, say that an ideal $I$ of a $k$-algebra $A$ is \emph{essential} if $I=\erad I$.
\end{defi}
\begin{rem} It is obvious that a maximal ideal $\m$ of $A$ is essential in the present sense if and only if it is in the sense of Definition \ref{EQUIRAD}, so non-trivial essential ideals of $A$ have points rational in $k$.
\end{rem}

\noindent We begin with the case of points : recall that if $(a_1,\ldots,a_n)\in k^n$ and $k[x_1,\ldots,x_n]=k[X_1,\ldots,X_n]/(X_1-a_1,\ldots,X_n-a_n)$, in $k[\ov x]$ we have $x_i=a_i$ for each $i$ and the structural morphism $k\to k[\ov x]$ is an isomorphism, so $(X_1-a_1,\ldots,X_n-a_n)$ is an essential maximal ideal.

\begin{lem}\label{BIJPNT}
For every $n\in\NN$, the map $P\in k^n\mapsto Ker(e_P)$ is a bijection between the points of the affine $n$-space and the essential maximal ideals of $A=k[X_1,\ldots,X_n]$, which have thus all the form $(X_1-a_1,\ldots,X_n-a_n)$ for $a_1,\ldots,a_n\in k$ (and the reciprocal bijection is given by $\m\mapsto \ms Z(\m)$).
\end{lem}
\begin{proof}
If $P=(a_1,\ldots,a_n)$, we have $(X_1-a_1,\ldots,X_n-a_n)= Ker(e_P)$ by what precedes; if $Q=(b_1,\ldots,b_n)$ and $Q\neq P$, it follows that $Ker(e_P)\neq Ker(e_Q)$, and the map is injective. As for surjectivity, if $\m$ is an essential maximal ideal of $A$, we have $\m=\erad \m$ by definition, so by Theorem \ref{NSREL} we have $\ms I(\ms Z(\m))=\erad\m=\m\neq A$, whence $\ms Z(\m)\neq\emptyset$ and there exists a zero $P$ of $\m$ in $k^n$, so that $Ker(e_P)\subs \m$; by maximality of $Ker(e_P)$, we get $Ker(e_P)=\m$ and the map is surjective. 
\end{proof}

\noindent In general, irreducible subvarieties have the same fruitful characterisation as in algebraically closed fields. Recall that if $V$ is an affine subvariety and $f\in k[V]$, we note $\ms Z_V(f)=\{P\in V : f(P)=0\}$.

\begin{lem}\label{CARIRRAFF}
If $V\subs k^n$ is an affine subvariety, then $V$ is irreducible if and only if $\Gamma(V,\O_V)$ is an integral domain, if and only if $k[V]$ is an integral domain.
\end{lem}
\begin{proof}
Suppose $V$ is irreducible and $f,g\in J(V):=\Gamma(V,\O_V)$ are such that $fg=0$ : we have $V=\ms Z_V(f)\cup\ms Z_V(g)$ and as $V$ is irreducible and $\ms Z_V(f),\ms Z_V(g)$ are closed, we get $V=\ms Z_V(f)$ or $V=\ms Z_V(g)$, i.e. $f=0$ or $g=0$, and $J(V)$ is an integral domain, as well as $k[V]$, which embeds into it. Next, suppose $k[V]$ is an integral domain, and let $V=V_1\cup V_2$, with $V_1=\ms Z_V(I_1)$ and $V_2=\ms Z_V(I_2)$ and distinguish two cases : if $I_1=(0)$, then $V=V_1$, whereas if $I_1\neq (0)$, there exists $f\in I_1$, $f\neq 0$; for every $P\in V$ and $g\in I_2$ we now have $fg(P)=0$, so $fg=0$ and as $k[V]$ is integral, we have $g=0$, and therefore $I_2=(0)$ and $V=V_2$ : $V$ is irreducible.
\end{proof}

\noindent Let $I$ be an ideal of $k[X_1,\ldots,X_n]$ : we have $\ms I(\ms Z(I))=\erad I$ by the \"Aquinullstellensatz, so $\ms Z(I)=\ms Z(\erad I)$, and thus every algebraic set is the zero set of an essential ideal. The correspondance is thus given as the following

\begin{prop}
The map $I\mapsto \ms Z(I)$ induces an order-reversing bijection between essential ideals of $k[\ov X]$ and algebraic subsets of $k^n$, between essential prime ideals and irreducible algebraic sets, and between essential maximal ideals and points of $k^n$.
\end{prop}
\begin{proof}
Suppose $I,J$ are essential and $\ms Z(I)=\ms Z(J)$ : we have $\ms I(\ms Z(I))=\ms I(\ms Z(J))$, so $I=\erad I=\erad J=J$ by what precedes, and the map is injective on essential ideals. If $V=\ms Z(I)\subs k^n$ is an algebraic subset, we have seen that $V=\ms Z(\erad I)$ so the map is surjective on essential ideals, it is a bijection. By Lemma \ref{CARIRRAFF}, an essential ideal $I$ is prime if and only if $\ms Z(I)$ is irreducible, which proves the second part of the statement. Finally, an essential ideal $I$ is maximal if and only if it is an essential maximal ideal, if and only if $\ms Z(\m)$ is a point by Lemma \ref{BIJPNT}.
\end{proof}
\begin{rem}
The picture may be completed as usual by a description of the topological closure $\ov S=\ms Z(\ms I(S))$ of any subset $S\subs k^n$, and by the relativisation of the correspondance to any affine subvariety.
\end{rem}

\noindent As in the classic case, the coordinate algebra functor $k[-]$ provides a duality if we replace finitely presented reduced $k$-algebras by their generalisation :

\begin{prop}\label{AFFDUAL}
The functor $k[-]$ is an duality between the categories of affine algebraic subvarieties of $k$ and essential $k$-algebras of finite type.
\end{prop}
\begin{proof}
We focus on essential surjectivity, so let $A$ be an essential $k$-algebra of finite type, isomorphic to an algebra $k[\ov X]/I$, where $\ov X=(X_1,\ldots,X_n)$. Consider the affine algebraic subvariety $V:=\ms Z(I)\subs k^n$ : we have $k[V]:=k[\ov X]/\ms I(V)=k[\ov X]/\erad I$ (by Theorem \ref{NSREL}) $=k[\ov X]/I$ by Lemma \ref{CARSPEQUO} because $A$ is essential; in short, $A$ is isomorphic to $k[V]$, and $k[-]$ is a duality.
\end{proof}

\subsection{Essential $*$-algebras of finite $*$-type}\label{STARESS}
Theorem \ref{CARERAD} provides a practical characterisation of the equiradical in finitely generated $k$-algebras. Now canonical localisations of $k$-algebras appear in Lemma \ref{CARPRE}, and these will prove to be the \og good" algebras to work with as rings of sections of affine sheaves of regular fonctions, and more generally in EQAG.
The following concept captures the intrinsic algebraic properties of canonical localisations.

\begin{defi}
Say that a $k$-algebra $A$ is a \emph{$*$-algebra (over $k$)} if every element of $M_A$ is invertible in $A$.
\end{defi}
\begin{ex}\label{CORFUNEX}
For every irreducible affine algebraic subvariety $V\subs k^n$, the function field $k(V)$ is a $*$-algebra : if $D\in\ms D$ and $\ov f/g\in k(V)$, we have $D(\ov f/g)=D^\#(\ov f/g,1)=(1/g^d)D^\#(\ov f,g)$, and as $D^\#(\ov f,g)\neq 0$ by Lemma \ref{CARSPESIG}, we have $D(\ov f/g)\in k(V)^\xx$. 
\end{ex}

\noindent More explicitly, a $k$-algebra $A$ is a $*$-algebra if and only if for all $D(\ov x)\in\ms D$ and $\ov a\in A$, there exists $b\in A$ such that $b.D(\ov a)=1$. If $k$ is algebraically closed, a $*$-algebra is thus exactly a $k$-algebra, but in general the two concepts do not coincide, as illustrated by Remark \ref{REMCANLOC}(iii) and the following 

\begin{lem}\label{CARSTAR}
If $A$ is a $k$-algebra and $l_M:A\to A_M$ its canonical localisation, then $A$ is a $*$-algebra if and only if $l_M$ is an isomorphism. In particular, $A_M$ is a $*$-algebra.
\end{lem}
\begin{proof}
If $A$ is a $*$-algebra, then for every $k$-morphism $f:A\to B$ with $f(M_A)\subs B^\xx$, there exists a unique $g:A\to B$ such that $g1_A=f$, so $1_A$ has the universal property of $l_M$, which is therefore an isomorphism. Conversely, it suffices to show that $A_M$ in general is a $*$-algebra. Let thus $D(\ov x)\in \ms D$ and $\ov a/m\in A_M$ an appropriate tuple : $m$ has the form $D_1(\ov a_1)$ for $D_1\in \ms D$ and $D(\ov a/m)=D^\#(\ov a/m,1)=(1/m^d)D^\#(\ov a,m)$. Let $D_2(\ov x,\ov x_1)=D^\#(\ov x,D_1(\ov x_1))$ : if $\ov b,\ov b_1\in k$, we have $D_1(\ov b_1)\neq 0$, hence $D^\#(\ov b,D_1(\ov b_1))\neq 0$ (otherwise $D(\ov b /D_1(\ov b_1))=(1/D_1(\ov b_1)^d)D^\#(\ov b,D_1(\ov b_1))=0$), so $D_2\in \ms D$, and therefore $m^dD(\ov a/m)=D_2(\ov a,D_1(\ov a_1))\in A_M^\xx$, whence $D(\ov a/m)\in A_M^\xx$ : $A_M$ is a $*$-algebra. 
\end{proof}
\begin{rem}
i) The key ingredient of the proof is borrowed from \cite{B-H}, Theorem 2.1.\\
ii) By the properties of localisation, every morphism $\phi:A\to B$ of $k$-algebras has a canonical localication $\phi_M:A_M\to B_M$. Canonical localisation is thus a functor from $k$-algebras to $*$-algebras, left adjoint to the forgetful functor. This last category is locally finitely presentable (see \cite{LPAC}), but we will not go into the category-theoretic details here, though we will use the notion of a $*$-algebra of finite type (as such). 
\end{rem}

\noindent The following very simple concept, inspired by the first order theory of quasivarieties (Sections 9.1 and 9.2 of \cite{HMT}), generalises reduced algebras of finite type over (algebraically closed) fields.

\begin{defi}\label{SPEALG}
Say that a $k$-algebra $A$ is \emph{essential} if $A$ embeds as a $k$-algebra into a power of $k$. 
\end{defi}


\begin{lem}\label{CARSPEQUO}
If $A$ is a $k$-algebra and $I$ an ideal of $A$, then $A/I$ is essential if and only if $I$ is essential.
\end{lem}
\begin{proof}
Suppose $A/I$ is essential : there exists a set $S$ and an embedding $\phi:A/I\into k^S$ of $k$-algebras. If $a\in A-I$, there exists $s\in S$ such that $p_s\circ \phi(a+I)\neq 0$, where $p_s: k^S\to k$ is the $s$-th projection. It follows that $a\notin  \m:=Ker(p_s\circ \phi\circ \pi_I)$, with $\pi_I:A\onto A/I$ the canonical projection; as $\m$ is essential, we get $a\notin \erad I$, so $I=\erad I$. Conversely, if $I=\erad I$, then by definition of $\erad I$ the quotient $A/I=A/\erad I$ embeds into $k^S$, where $S$ is the set of essential maximal ideals containing $I$, so $A/I$ is essential.
\end{proof}

\noindent The following property of essential ideals is an analogue of the characterisation of radical ideals.

\begin{lem}\label{CARSPESIG}
If $I$ is an essential ideal of a $k$-algebra $A$, then for all $D(\ov x)\in \ms D$, $\ov a,b\in A$ and $m,n\in\NN$ such that $b^mD^\#(\ov a,b^n)\in I$, we have $b\in I$. Conversely, if this property holds and $A$ is \emph{finitely generated}, then $I$ is essential. 
\end{lem}
\begin{proof}
Suppose $I$ is essential : by Lemma \ref{CARSPEQUO}, there exists an embedding $\phi:A/I\into k^S$ for a set $S$, so let $D(\ov x)\in \ms D$, $\ov a,b\in A$ and $m,n\in\NN$ be such that $b^mD^\#(\ov a,b^n)\in I$ : write $\ov a_I=\ov a+I$, $b_I=b+I$, for each $s\in S$ we have $\phi b_I(s)^mD^\#(\phi\ov a_I(s),\phi b_I(s)^n)=0$ in $k$. If $\phi b_I(s)\neq 0$, then $D(\phi\ov a_I(s)/\phi b_I(s)^n)=\phi b_I(s)^{-nd}D^\#(\phi \ov a_I(s),\phi b_I(s)^n)=0$ in $k$, which contradicts the definition of $\ms D$, so $\phi b_I(s)=0$ for all $s$, whence $\phi(b_I)=0$ and therefore $b_I=0$, i.e. $b\in I$.
Conversely, if the property holds and $b\notin I$, let $D(\ov x)\in \ms D$, $\ov a\in A$ and $n\in\NN$ such that $D([\ov a]/[b^n])=0$ in $B:=(A/I)_{[b]}$ : we get $D^\#([\ov a],[b^n])=0$, thus $\exists m\in\NN$ such that $b^m D^\#(\ov a,b^n)=0$ in $A/I$, whence $b\in I$ by hypothesis. This contradicts our assumption, so we must have $D([\ov a]/[b^n])\neq 0$, therefore $0\notin M_B$, whence $B_M\neq 0$ and by Lemma \ref{CARPRE} there exists a $k$-morphism $B_M\to k$ : the kernel $\m$ of the composite morphism $A\to (A/I)_{[b]}\to B_M\to k$ is essential and contains $I$ but not $b$, so $b\notin \erad I$. We conclude that $I=\erad I$, i.e. that $I$ is essential.
\end{proof}
\begin{rem}\label{HOMSIGN}
Essential ideals may as well be characterised in \emph{finitely generated} $k$-algebras by the \emph{homogeneous signature} of $k$ (Definition \ref{DEFHOMSGN}), which we will rather introduce in Section \ref{SPECALGRAD}, where we get rid of the finiteness hypotheses.
\end{rem}

\noindent Canonical localisations of essential finitely generated $k$-algebras remain essential.

\begin{lem}\label{SIMPESS}
If $A$ is an essential $k$-algebra, then the elements of $M_A$ are simplifiable.
\end{lem}
\begin{proof}
By hypothesis, there exists an embedding $\phi:A\into k^S$ into a power of $k$. Let $m\in M_A$ and $f\in A$ such that $mf=0$ : for each $s$, we have $\phi(mf)(s)=0$, so that $\phi(f)=0$ because $\phi(m)(s)\in k^\xx$ for each $s$. It follows that $f=0$, so $m$ is simplifiable.
\end{proof}

\begin{prop}\label{ESSTAR}
If $A$ is an essential $k$-algebra of finite type, then the canonical localisation $l_M:A\to A_M$ is injective and $A_M$ is essential.
\end{prop}
\begin{proof}
By Lemma \ref{SIMPESS}, $l_M$ is a localisation by simplifiable elements, it is therefore injective.
Secondly, let $\phi:A\into k^S$ be an embedding into a power of $k$ :  every $m\in M_A$ has $\phi(m)\in (k^S)^\xx$, and by universal property of $A_M$ there exists (a unique $k$-morphism) $\psi:A_M\to k^S$ such that $\psi l_M=\phi$, and $\psi$ is injective as $l_M$ is.
\end{proof}

\noindent All this and the forthcoming characterisation of local sections of affine sheaves of regular functions shows that $*$-algebras obtained by canonical localisation of finitely generated $k$-algebras play a special role in the theory.

\begin{defi}\label{FINSTARTYP}
Say that a $*$-algebra $A$ has \emph{finite $*$-type} if it is isomorphic to the canonical localisation of a finitely generated $k$-algebra.
\end{defi}
\begin{rem}
The $*$-algebras of finite $*$-type are exactly the finitely generated $*$-algebras in the category of $*$-algebras, but they are not finitely generated \emph{as $k$-algebras} in general.
\end{rem}

\begin{cor}\label{CARSPESIG'}
If $A$ is a $*$-algebra of finite $*$-type, then $A$ is essential if and only if for all $D\in\ms D$, $\ov a,b\in A$ and $m,n\in\NN$, we have $b=0$ whenever $b^m D^\#(\ov a,b^n)=0$.
\end{cor}
\begin{proof}
We may assume that $A=B_M$, with $B$ a $k$-algebra of finite type. The direct sense is a particular case of Lemma \ref{CARSPESIG}, so assume the property holds and let $D\in\ms D$, $\ov a,b\in B$, $m,n\in\NN$ be such that $b^m D^\#(\ov a,b^n)=0$ in $B$ : this is also true in $A$, so $b/1=0$ and thus $\mu b=0$ in $B$ for some $\mu\in M_A$; by Lemma \ref{SIMPESS} we get $b=0$, and by Lemma \ref{CARSPESIG} we conclude that $B$ is essential. By Proposition \ref{ESSTAR}, $A_M$ is essential.
\end{proof}

\section{Affine algebraic equivarieties}\label{SEQVAR}
In this section we study the basic geometric objects of EQAG, which are the affine algebraic equivarieties. Their archetypes are the affine subvarieties $V\subs k^n$ with their sheaves of regular functions, which rings of global sections $\Gamma(V,\O_V)$ are no longer, in the general case, isomorphic to the coordinate rings $k[V]$, but rather to their canonical localisations (Theorem \ref{MAG.3.6.a}). This means that the duality of Proposition \ref{AFFDUAL} cannot be extended to the abstract geometric setting, and must be replaced with a duality with essential $*$-algebras of finite $*$-type.

\subsection{Sheaves of regular functions}\label{AFFSUB}
Thanks to Theorem \ref{NSREL} and the algebraic concepts introduced in Section \ref{STARESS}, we are now able to investigate the sheaf of regular functions on an affine algebraic subvariety $V$ of $k^n$ - which has a more complex structure than in classic algebraic geometry - while giving an algebraic interpretation of it in the present setting.
If $h,h'\in k[V]$, we have $$(*)\ D_V(h)\subs D_V(h') \Iff \ms Z_V(h)\supseteq \ms Z_V(h') \Iff \erad{(h)}\subs \erad{(h')} \Iff h\in \erad{(h')} \Iff \Sigma_h\cap (h')\ne\emptyset$$ (by definition of $\erad{}$ and Theorem \ref{CARERAD}). Now let $g,h\in k[V]$ and $\alpha\in \Sigma_h$ : if $P\in D_V(h)$, one checks that $h(P)\alpha(P)\neq 0$, so $g/h\alpha$ defines a regular function on $D_V(h)$, i.e. an element of $\O_V(D_V(h))$. We record the obvious

\begin{lem}\label{MAG.3.5}
The map $P\in D_V(h)\mapsto g(P)/h\alpha(P)$ is zero on $D_V(h)$ if and only if $gh=0$ in $k[V]$.
\end{lem}

\noindent Although the next lemma should be considered as folklore, we include it for the sake of completeness.

\begin{lem}\label{ISOSUB}
For any affine algebraic subvariety $V\subs k^n$, any basic open subset of $V$ is isomorphic, as a locally ringed space in $k$-algebras, to an affine algebraic subvariety of $k^{n+1}$.
\end{lem}

\noindent A first topological consequence of the ANS is the following essential

\begin{lem}\label{COMPSUB}
Every affine algebraic subvariety is compact for the Zariski topology.
\end{lem}
\begin{proof}
Let $V\subs k^n$ be any such subvariety and $V=\bigcup_I D_V(f_i)$ a basic open cover : we have $\emptyset=\bigcap_I \ms Z_V(f_i)=\ms Z_V(\sum_I k[V] f_i)$, so $k[V]=\ms I(\emptyset)=\ms I(\ms Z_V(\sum_I k[V] f_i))=\erad{\sum_I k[V]f_i}$ by Theorem \ref{NSREL}, and thus by Theorem \ref{CARERAD} there exists $m\in \Sigma_1=M_{k[V]}$, a finite subset $I_0$ of $I$ and $(a_i:i\in I_0)$ in $k[V]$ such that $m=\sum_{I_0} a_i f_i$, therefore $\emptyset=\ms Z_V(m)\supseteq\ms Z_V(\sum_{I_0} k[V]f_i)=\bigcap_{I_0} \ms Z_V(f_i)$, whence $V=\bigcup_{I_0} D_V(f_i)$.
\end{proof}

\noindent So far we have defined a natural map $k[V]_{<h>}\to \Gamma(D_V(h),\O_V)$, $g/\alpha\mapsto [P\mapsto g(P)/\alpha(P)]$, which is an injective $k$-morphism : if the member on the right is zero, then the regular map defined by $gh/h\alpha$ is zero on $D_V(h)$, so $gh^2=0$ in $k[V]$ by Lemma \ref{MAG.3.5} and $g/\alpha=gh^2/h^2\alpha=0$ in $k[V]_{<h>}$. The following theorem is our core structure result, bringing together the preceding algebraic theory and the affine geometric theory, and is inspired by \cite{MAG}, Proposition 3.6(a).

\begin{thm}\label{MAG.3.6.a}
The morphism $k[V]_{<h>}\to \Gamma(D_V(h),\O_V)$ is an isomorphism. In particular, $\O_V$ is a sheaf of essential $*$-algebras over $V$.
\end{thm}
\begin{proof}
As for the first assertion, it only remains to prove that the morphism is surjective. Let $f\in \Gamma(D_V(h),\O_V)$ : there exists an open cover $D_V(h)=\bigcup_i U_i$, as well as $g_i,h_i\in k[V]$ for each $i$, such that for every $i$ we have $f|_{U_i}\equiv g_i/h_i$, and we may assume that $U_i=D_V(a_i)$ is a basic open for each $i$ : we have $D_V(a_i)\subs D_V(h_i)$ and by $(*)$, for every $i$ there exists $\alpha_i\in \Sigma_{a_i}$ and $g_i'\in k[V]$ with $\alpha_i=g_i'h_i$; on $D_V(a_i)$, $f$ is represented by $g_ig_i'a_i/a_ih_ig_i'=g_ig_i'a_i/a_i\alpha_i$ : replacing $g_i$ by $g_ig_i'a_i$, and $h_i$ by $a_i\alpha_i$, as $D_V(a_i)=D_V(a_i\alpha_i)$ we may suppose that $U_i=D_V(h_i)$ for all $i$ and now have $D_V(h)=\bigcup_i D_V(h_i)$, with $f$ represented on $D_V(h_i)$ by $g_i/h_i$. By Lemmas \ref{ISOSUB} and \ref{COMPSUB}, $D_V(h)$ is compact so we may assume that this cover is finite, and as the functions represented by $g_i/h_i$ and $g_j/h_j$ on $D_V(h_i)\cap D_V(h_j)=D_V(h_ih_j)$ are equal, we have $(g_ih_j-g_jh_i)/h_ih_j\equiv 0$ on $D_V(h_ih_j)$, whence by Lemma \ref{MAG.3.5} $h_ih_j(g_ih_j-g_jh_i)=0$ in $k[V]$, i.e. $h_ih_j^2g_i=h_i^2h_jg_j$. Now we have $D_V(h)=\bigcup_{i=1}^m D_V(h_i)=\bigcup_{i=1}^m D_V(h_i^2)$, so $\ms Z_V((h))=\ms Z_V((h_1^2,\ldots,h_m^2))$, whence $h\in \ms I(\ms Z_V((h_1^2,\ldots,h_m^2))=\erad{(h_1^2,\ldots,h_m^2)}$ by the \"Aquinullstellensatz \ref{NSREL} again, and by Theorem \ref{CARERAD} there exist $\alpha\in \Sigma_h$ and $a_i\in k[V]$ such that $\alpha=\sum_{i=1}^m a_i h_i^2$ : we want to show that $f$ is represented on $D_V(h)$ by $(\sum_i a_ig_ih_i) / \alpha$. Let $P\in D_V(h)$ : for each $j$ such that $P\in D_V(h_j)$, we have $h_j^2\sum_{i=1}^m a_ig_ih_i=\sum_{i=1}^m a_i g_j h_j h_i^2=g_j h_j \alpha$ by what precedes, and $f$ is represented on $D_V(h_j)$ by $g_j/h_j$, so $fh_j\equiv g_j$ on $D_V(h_j)$, on which therefore we have $fh_j^2\alpha\equiv g_jh_j\alpha\equiv h_j^2\sum_{i=1}^m a_ig_ih_i$ as maps. As $h_j(P)^2\neq 0$, on $D_V(h_j)$ we have $f\alpha\equiv\sum_i a_ig_ih_i$ as maps, so that $f$ is represented on $D_V(h_j)$ by $(\sum_i a_i g_i h_i)/ \alpha$, and as this is true for every $j$, this is true on $D_V(h)$, so finally the morphism $k[V]_{<h>}\to \Gamma(D_V(h),\O_V)$ is surjective, it is an isomorphism. As for the second assertion, for every open subset $U\subs V$, we have $U=\bigcup\{ D_V(f) : D_V(f)\subs U, f\in k[V]\}$, and therefore $\O_V(U)=\proj_{D(f)\subs U} \O_V(D(f))$. Now by what precedes each $\ms O_V(D(f))$ is a $*$-algebra by Lemmas \ref{CARUNLOC} and \ref{CARSTAR}, and an essential algebra as well by Lemma \ref{ISOSUB}; as $*$-algebras and essential algebras are clearly closed under projective limits, $\ms O_V(U)$ is an essential $*$-algebra.
\end{proof}

\begin{cor}\label{CARREGMAP3}
For every affine algebraic subvariety $V\subs k^n$, the $k$-algebra $\Gamma(V,\O_V)$ of everywhere regular functions on $V$ is isomorphic to $k[V]_M$, under the natural $k$-morphism $k[V]_M\to \Gamma(V,\O_V)$.
\end{cor}
\begin{proof}
By Theorem \ref{MAG.3.6.a}, we have $\Gamma(V,\O_V)=\Gamma(D_V(1),\O_V))\cong k[V]_{\<1\>}=k[V]_M$. \end{proof}
\begin{rem}\label{REMCAN}
i) This is the generalisation over an arbitrary field of the characterisation of the ring of global sections of the sheaf of regular functions on an affine algebraic subvariety (if $k$ is algebraically closed, this algebra is essentially $k[V]$).\\
ii) In general $k[V]_M$ is bigger than $k[V]$ (keeping in mind Remark \ref{REMCANLOC}(iii)), so $\Gamma(V,\O_V)$ is not always isomorphic to $k[V]$. 
\end{rem}

\subsection{Global sections of affine equivarieties}
In order to build up the theory of intrinsic geometrical objects associated with EQAG, we introduce a natural category in which we find them among the locally ringed spaces in $k$-algebras. We refer for instance the reader to Chapter II of \cite{HAG} about generalities on sheaves and locally ringed spaces. We focus here on (equiresidual) \emph{affine} algebraic varieties, and will expand to the projective setting in \cite{EQAG2}. 

\begin{defi}\label{DEFEQVAR}
i) An \emph{equiresidual variety over $k$}, or \emph{equivariety over $k$} for short, is a locally ringed space in $k$-algebras $(V,\O_V)$, such that for every $P\in V$, there exists an open neighbourhood $U$ of $P$ in $V$ for which $(U,\O_V|_U)$ is isomorphic to an affine algebraic subvariety with its sheaf of regular functions.\\
ii) We will say that a locally ringed space in $k$-algebras $(V,\O_V)$ is \emph{concrete} if for every open subset $U\subs V$, the $k$-algebra $\O_V(U)$ is a subset of the $k$-algebra $k^U$ of all functions $U\to k$. Say that a map $f:V\to W$ between any such is \emph{regular} if it is continuous and for all open subset $U\subs W$ and $s\in \O_W(U)$, we have $f_U^\#(s):=s\circ f\in \O_V(f^{-1} U)$ : we implicitly consider $f$ as the sheaf morphism $(f,f^\#)$, where $f^\#_U:\O_W(U)\mapsto f_*\O_V(U)$.
\end{defi}
\begin{rem}
By definition, for every open subset $U$ of a \emph{concrete equivariety} $V$ the members of $\O_V(U)$ are \emph{continuous} functions.
\end{rem}

\noindent As morphisms between equivarieties we consider all morphisms of locally ringed spaces in $k$-algebras. For the sake of simplicity, we restrict ourselves to \emph{concrete} equivarieties, which is harmless thanks to the following proposition.

\begin{prop}\label{CONCEQ}
Every morphism of concrete equivarieties is a regular map (and reciprocally), and every equivariety over $k$ is naturally isomorphic to a concrete equivariety.
\end{prop}
\begin{proof}
As for the first claim, if $(f,f^\#):V\to W$ be a morphism of concrete equivarieties over $k$, by hypothesis $f$ is continuous and we want to show that for every open $U\subs W$, $f^\#_U(s)=s\circ f$. 
For $P\in U$, write $Q=f(P)$ : the induced $k$-morphism $f^\#_P:\O_{W,Q}\to \O_{V,P}$ is local, and by definition we have $f^\#_P([s,U]_Q)=[f^\#_U(s),f^{-1}U]_P$, so let $i_P:k\to\O_{V,P}$ and $i_Q:k\to \O_{W,Q}$ be the structural morphisms and $j_P:k\cong\ov{\O_{V,P}}$ and $j_Q:k\cong \ov{\O_{W,Q}}$ the associated residual structural isomorphisms, so that $f^\#_P\circ i_Q=i_P$ and $\ov{f^\#_P}\circ j_Q=j_P$ with $\ov{f^\#_P}:\ov{\O_{W,Q}}\to\ov{\O_{V,P}}$ the residual $k$-isomorphism. We have $j_P(f^\#_U(s)(P))=\ov{[f^\#_U(s),f^{-1}U]_P}=\ov{f^\#_P}(\ov{[s,U]_Q})=\ov{f^\#_P}\circ j_Q(s(f(P)))=j_P(s\circ f(P))$, whence $f^\#_U(s)(P)=s\circ f(P)$, so $f^\#_U(s)=s\circ f$, and thus $(f,f^\#)$ is regular.
As for the second claim, an equivariety $(V,\O_V)$ over $k$ has a basis of \og affine" open subsets : for $U\subs V$ one such there exists an isomorphism $\phi:U\cong V_0\subs k^n$ with an affine subvariety, and we let $\O_V'(U):=\{s\circ \phi : s\in\O_{V_0}(V_0)\}$, a $k$-algebra of functions $U\to k$. This is well defined, for if $\psi:U\cong W_0\subs k^m$ is another isomorphism and $s\circ \phi\in \O_V'(U)$, we have $t:=s\circ \phi\circ \psi^{-1}:W_0\to U\to V_0\to k$, and as $\phi\circ \psi^{-1}$ is an isomorphism and $s\in \O_{V_0}(V_0)$, by the first part of the proof we have $t\in \O_{W_0}(W_0)$, because $\phi\circ \psi^{-1}$ is regular; it follows that $t\circ \psi=s\circ \phi$, so $s\circ\phi\in \{t\circ \psi : t\in \O_{W_0}(W_0)\}$ and by symmetry, we get $\{s\circ \phi : s\in \O_{V_0}(V_0)\}=\{t\circ \psi : t\in \O_{W_0}(W_0)\}$, so that $\O_V'(U)$ is well defined. Now define a sheaf $\O'_V$ of continuous functions as follows : if $U\subs V$ is open, write $U=\bigcup_i U_i$ as the union of \emph{all} affine open subsets $U_i$ of $U$ and put $\O_V'(U):=\{f:U\to k :\forall i,f|_{U_i}\in \O_V'(U_i)\}$ : $\O'_V$ : if $U'\subs U$ is open, the cover $U'=\bigcup_j U'_j$ of $U'$ by all affine open subsets has each $U'_j$ being one of the $U_i$'s, and $\O'_V$ is well defined. It remains to prove that $\O_V'\cong \O_V$ and it suffices to describe an isomorphism on the affine open subsets of $V$, so if $U\subs V$ is any such with $\phi:U\cong V_0\subs k^n$ a local chart, to $s\in \O_V(U)$ we associate $(\phi^\#_{V_0})^{-1}(s)\circ \phi\in \O_V'(U)$ and to every $t\circ \phi\in \O_V'(U)$ we associate $\phi_{V_0}^\#(t)\in \O_V(U)$; as $\phi$ is an isomorphism, this is well defined, and both maps are obviously mutually inverse $k$-isomorphisms. 
\end{proof}

\begin{defi}
An \emph{affine (algebraic) equivariety over $k$} is an equivariety over $k$ which is isomorphic to an affine algebraic subvariety of $k$.
\end{defi}

\noindent We note $EVar^a_k$ the category of \emph{concrete} affine equivarieties over $k$, with arrows the regular morphisms; as we have seen in the introduction, any regular morphism of affine subvarieties is naturally a regular morphism of equivarieties.
If $f:V\to W$ is a morphism in $EVar^a_k$, we have the $k$-algebra morphism $f^\#_W:s\in \O_W(W)\mapsto s\circ f\in \O_V(V)$, and we let $J(f)=f^\#_W$ : this defines a functor $J$ from $EVar^a_k$ to the dual of the category of $k$-algebras. Now the duality of Proposition \ref{AFFDUAL} \emph{cannot}, as in the classic case, be extended to $EVar^a_k$ using $J$, because by Corollary \ref{CARREGMAP3}, for $V\subs k^n$ we only have $J(V,\O_V)\cong k[V]_M$. This means we have to search for another duality, which is provided by the following : 

\begin{defi}
Say that a $*$-algebra $A$ over $k$ is \emph{affine}, if it is isomorphic to an algebra of the form $\Gamma(V,\O_V)$, for $V$ an affine algebraic subvariety of $k$. 
\end{defi}
\begin{rem}
By Corollary \ref{CARREGMAP3}, an affine $*$-algebra is of finite $*$-type (Definition \ref{FINSTARTYP}). 
\end{rem}

\noindent Now the functor $J$ has values in the dual of $*Aff_k$, the category of affine $*$-algebras. We have already encountered these in the following form : 

\begin{prop}\label{AFFSTAR}
A $k$-algebra $A$ is an affine $*$-algebra over $k$ if and only if it is an essential $*$-algebra of finite $*$-type over $k$.
\end{prop}
\begin{proof}
Let $V\subs k^n$ be an affine algebraic subvariety of $k$ such that $A\cong \Gamma(V,\O_V)$ : we have $\Gamma(V,\O_V)\cong k[V]_M$ by Corollary \ref{CARREGMAP3} and as $k[V]$ is special by definition, so is
$k[V]_M$ by Proposition \ref{ESSTAR}. Now the surjective morphism $k[\ov X]\onto k[V]$ gives by canonical localisation a surjective morphism $k[\ov X]_M\onto k[V]_M$, so $A$ is an essential $*$-algebra of finite $*$-type. Conversely, if $\phi:k[X_1,\ldots,X_n]_M\onto A$ is a surjective $k$-morphism, and $B=\phi(k[X_1,\ldots,X_n])$, we have $A\cong B_M$ and as $A$ is essential, so is $B$ as a subalgebra. By Proposition \ref{AFFDUAL}, $B$ is isomorphic to the coordinate algebra $k[V]$ of the affine algebraic subvariety $V=\ms Z(I)$ with $I=Ker(\phi|_{k[\ov X]})$. In particular, we have $A\cong B_M\cong \Gamma(V,\O_V)$ by \ref{CARREGMAP3} again, therefore $A$ is an affine $*$-algebra over $k$.
\end{proof}

\subsection{The essential spectrum and the affine duality}\label{ESSPECDUAL}
\noindent So far we have passed from varieties to algebras, and rather than showing directly that the global sections functor $J:EVar^a_k\to (*Aff_k)^o$ is a duality, we introduce a natural spectrum functor which goes the other way.

\begin{defi}
If $A$ is a $k$-algebra, say that the \emph{essential spectrum of $A$} is the set noted $Sp_k(A)$ of all $k$-morphism $\phi:A\to k$. The \emph{Zariski topology on $Sp_k(A)$} has basic open sets of the form $D(f)=\{\phi\in Sp_k(A) : \phi(f)\neq 0\}$, for $f\in A$. 
\end{defi}

\begin{defi}
If $A$ is a $k$-algebra and $X=Sp_k(A)$, say that a function $f:U\to k$ is \emph{regular} on an open $U\subs X$, if for every $\phi\in U$ there exists an open neighbourhood $U_\phi\subs U$ of $\phi$ and $a,b\in A$  such that for every $\psi\in U_\phi$ we have $f(\psi)=\psi(a)/\psi(b)$.
This defines the \emph{sheaf $\O_X$ of regular functions on $X$}.
\end{defi}

\begin{rem}
A regular function $f:U\to k$ is continuous for the Zariski topologies.
\end{rem}

\begin{lem}\label{LEMTIG2}
Let $A$ be a $k$-algebra, $X=Sp_k(A)$ its essential spectrum and $\m=\{f\in \Gamma(X,\O_X) : f(x)=0\}$ for $x\in X$. We have a natural isomorphism $\O_{X,x}\cong J(X)_\m$ (where $J(X)=\Gamma(X,\O_X)$).
\end{lem}
\begin{proof}
Let $[f]\in \O_{X,x}$, $U\ni x$ open and $a,b\in A$ such that $f|_U\equiv a/b$, and let $\phi([f]):= a/b\in J(X)_\m$ ($a$ and $b$ identified with their images under $A\to J(X)$). If $O\ni x$ is open and $f|_O\equiv c/d$, there exists $W=D(e)\subs O\cap U$ such that $a/b \equiv c/d$ on $W$, $D(e)=D(bde)\subs D(b)\cap D(d)$ and $e\in J(X)-\m$. We thus have $(ade-bce)/bde|_W\equiv 0$ as a section of $\O_X$, whence $ade-bce\equiv 0$ on $D(e)$; as $bde\equiv 0$ on $J(X)-D(e)$, we get $(ade-bce)bde\equiv 0$ in $J(X)$, and as $bde\notin\m$ we get $ade-bce=0$ in $J(X)_\m$, whence $a/b=c/d$, and $\phi$ is well defined, obviously an $k$-morphism.
Now let $f/g\in J(X)_\m$, $a,b,c,d\in A$ and $U\ni x$ such that $f/g|_U\equiv (a/b)/(c/d)$ : we have $f/g=\phi([ad/bc])$ and $\phi$ is surjective. As for injectivity, if $[f]\in \O_{X,x}$ as before with $f|_U\equiv a/b$ and $\phi([f])=0$, let $g\in J(X)-\m$ with $ga=0$ in $J(X)$ : we have $[a/b,U]=[ag/bg,D(bg)]=0$ in $\O_{X,x}$, so $\phi$ is an isomorphism.
\end{proof}

\noindent If $f:A\to B$ is a morphism of affine $*$-algebras, the map $f^a:Y=Sp_k(B)\to X=Sp_k(A)$, $\psi\mapsto \psi\circ f$, is obviously continuous and if $U\subs X$ is open and $s\in \O_X(U)$, we have $s\circ f^a\in \O_Y((f^a)^{-1} U)$, using $f$ to \og translate" over $(f^{a})^{-1} U$ the local representations of $s$ over $U$. By Lemma \ref{LEMTIG2}, this means we have a functor $K:\ov f\mapsto f^a$ from $(*Aff_k)^o$ to the category of concrete locally ringed spaces in $k$-algebras.\\   
Now let $V\subs k^n$ be an affine subvariety, and consider the natural mapping $\Phi_V:V\to X=Sp_k(J(V))$, $P\mapsto [e_P:f\in J(V)\mapsto f(P)]$ : using the coordinate functions in $J(V)$ we see that $\Phi_V$ is a bijection, and using the isomorphism $J(V)\cong k[V]_M$, that it is a homeomorphism. 

\begin{prop}\label{SPESS2}
If $V$ is an affine subvariety, then $X=Sp_k(\Gamma(V,\O_V))$ is a concrete equivariety and the homeomorphism $\Phi_V$ is a regular isomorphism.
\end{prop}
\begin{proof}
Writing $\Phi=\Phi_V$, we only have to show that $\Phi$ is regular. If $U\subs X$ is open and $s\in\O_X(U)$, we have $\Phi^\#_U(s)=s\circ \Phi$ : let $P\in \Phi^{-1}U$, $U_P\subs U$ open and $a,b\in A=\Gamma(V,\O_V)$ such that $e_P\in U_P$ and $s|_{U_P}\equiv a/b$, we have $(s\circ \Phi)|_{\Phi^{-1} U_P}\equiv a/b$ as well, so $s\circ \Phi\in \O_V(\Phi^{-1}U)$. Now for each $P\in V$, the induced residual morphism $\Phi_P:\O_{X,e_P}\to \O_{V,P}$ is obviously local by Lemma \ref{LEMTIG2}, so $\Phi$ is a regular morphism of locally ringed spaces in $k$-algebras. With the same notations, assume $\Phi^\#_U(s)=0$ : for all $P\in \Phi^{-1}U$ we have $0=s\circ \Phi(P)=s(e_P)$, hence $s|_U\equiv 0$ and $\Phi^\#_U$ is injective. As for surjectivity, if $g\in \O_V(\Phi^{-1}U)$ for each $P\in\Phi^{-1}U$ represent $g$ on an open nieghbourhood $O_P\subs \Phi^{-1}U$ of $P$ by $a_P/b_P$ for $a_P,b_P\in k[V]$ and apply $\Phi$ : on the open $U_P=\Phi(O_P)\subs U$, $b_P$ does not vanish and we let $s:e_P\in U\mapsto a_P(P)/b_P(P)$. If $Q\in O_P$, as $g|_{O_P}\equiv a_P/b_P$ and $g|_{O_Q}\equiv a_Q/b_Q$, we have $s|_{U_P}\equiv a_P/b_P$ for all $P$, hence $s\in \O_X(U)$ and by definition we have $s\circ \Phi=g$, whence $\Phi$ is an isomorphism.
\end{proof}

\noindent If $A$ is an affine $*$-algebra, then $A\cong J(V)$ for $V$ an affine subvariety, and the induced isomorphism $X=Sp_k(A)\cong Sp_k(J(V))$ is regular by Proposition \ref{SPESS2} : this means that the functor $K$ has values in the full sub-category $EVar^a_k$ and we may now use it to dualise this one. To this end we let $f_A:a\in A\mapsto s\in JK(A)$ with $s\equiv a$ on $K(A)=Sp_k(A)$; the mapping $f_A$ is obviously a $k$-morphism. 

\begin{prop}\label{ISOVAR2}
The regular isomorphisms $\Phi_V:V\cong X=KJ(V)$, for $V$ a concrete affine equivariety, define a natural isomorphism $\Phi:Id\cong K\circ J$.
\end{prop}
\begin{proof}
As for naturality, it suffices to show that for every regular morphism $f:V\to W$ of concrete affine equivarieties, the following diagram commutes :
$$\begin{CD}
V @>f>> W\\
@V\Phi_{V} VV @VV\Phi_W V\\
X @>>g> Y,\end{CD}$$
where $g=KJ(f):e_P\in X=Sp_k(\Gamma(V,\O_V))\mapsto e_{f(P)}\in Y=Sp_k(\Gamma(W,\O_W))$ by definition. Now if $P\in V$, we have $\Phi_W\circ f(P)=e_{f(P)}=g(e_P)=g\circ \Phi_V(P)$, and $\Phi$ is a natural transformation.
As for $\Phi$ being an isomorphism, for $V$ in $Evar^a_k$ choose $f:V\cong W$ in the preceding diagram with $W$ an affine subvariety : by Proposition \ref{SPESS2}, $\Phi_W$ is a regular isomorphism; as $f$ and $KJ(f)$ also are, it follows that $\Phi_V$ itself is an isomorphism, and the proof is complete.
\end{proof} 

\begin{prop}\label{ISOALG}
The $k$-algebra morphisms $f_A:A\to JK(A)$, for $A\in *Aff^o_k$, define a natural isomorphism $f:Id\cong J\circ K$.
\end{prop}
\begin{proof}
Let $g:A\to B$ be a $k$-morphism of affine $*$-algebras and $(X,\O_X):=K(A)$, $(Y,\O_Y):=K(B)$ : we have the morphism $K\ov g:Y\to X$, $\psi\mapsto \psi\circ g$ and $JK\ov g:s\in \Gamma(X,\O_X)\mapsto s\circ K\ov g$, with $s\circ K\ov g:\psi\in Y\mapsto s(\psi\circ g)$. It follows that for each $a\in A$, we have $JK\ov g\circ f_A(a):\psi\mapsto f_A(a)\circ K\ov g (\psi)=f_A(a)(\psi\circ g)=\psi\circ g(a)=f_B(g(a))(\psi)$, whence $JK\ov g\circ f_A=f_B\circ \phi$, and $f$ is a natural transformation :
$$\begin{CD}
A @>g>> B\\
@Vf_A VV @VVf_B V\\
JK(A) @>>JK\ov g> JK(B).\end{CD}$$
Now assume that $A=\Gamma(V,\O_V)$ for $V$ an affine subvariety : if $f_A(a)=0$, for each character $e_P\in X$ we have $a(P)=0$, so $a=0$ and $f_A$ is injective. As for surjectivity, if $s\in\Gamma(X,\O_X)$ we have seen in Proposition \ref{SPESS2} that $\phi_V$ is regular, so that $s\circ \phi_V\in A$ : for $e_P\in X$ we have $f_A(s\circ\phi_V)(e_P)=e_P(s\circ \phi_V)=s\circ\phi_V(P)=s(e_P)$, so $f_A(s\circ\phi_V)=s$, and $f_A$ is surjective, it is an isomorphism. In general, for $B$ in $*Aff_k$ there exists an isomorphism $g:A=\Gamma(V,\O_V) \cong B$ and by naturality of the diagram, $f_B$ is an isomorphism.
\end{proof}

\noindent In this situation we have an \emph{adjoint equivalence} (\cite{CWM}, Theorem IV.4.1) so putting together Propositions \ref{ISOVAR2} and \ref{ISOALG}, we obtain the duality theorem between (concrete) affine equivarieties and affine $*$-algebras :

\begin{thm}\label{DUALAFF}
The global sections functor $J:EVar^a_k\to *Aff_k^o$ is a duality between the categories $EVar^a_k$ and $*Aff_k$, with adjoint the essential spectrum functor $K$.
\end{thm}

\part{Extending the Algebraic Universe}
In principle, in EQAG we are not looking for rational points of varieties in \emph{algebraic} extensions of $k$ (a point which is exhaustively clarified in \ref{SPEXTRI}). However, certain natural extensions of $k$ or even $k$-algebras, such as the field $k(V)$ of rational functions or its generalised version on any affine variety $V$, belong in some sense in the \og algebraic universe" of EQAG over $k$, which we therefore wish to extend in this second part.

\section{Special algebra}\label{SPECALGRAD}
\emph{Special} algebras were already introduced implicitly as \emph{essential} algebras in Definition \ref{SPEALG} in a finitary context, and they are generalised using the characterisation of Corollary \ref{CARSPESIG'}. They appear to be subtly connected to $*$-algebras and fundamentally to normic forms (Definition \ref{NORMFORM}), which lie at the root of the theory. Furthermore, they naturally allow us to deal with the application of EQAG to model-theoretic algebraic geometry, which is tackled in \cite{PAG}.

\subsection{The homogeneous signature of a field}\label{SECHOMSGN}
In Lemma \ref{CARSPESIG} and Corollary \ref{CARSPESIG'} appear homogenisations of elements of the algebraic signature, which leads us to the following analogue of \ref{ALGSIG} : 

\begin{defi}\label{DEFHOMSGN}
The \emph{homogeneous signature of $k$} is the set $\ms H$ of all non-zero homogeneous polynomials $P(\ov X,Y)\in k[X_i:i\in\NN]$ such that for all appropriate $\ov a\in k$, $0$ is the only zero of $P(\ov a,Y)$ in $K$.
\end{defi}

\begin{rem}\label{INCSIG0}
i) If $D(\ov x)\in\ms D$ and $\ov a,b\in k$ are such that $D^\#(\ov a,b)=0$ with $b\neq 0$, then $D(\ov a/b)=0$, which is impossible. It follows that the set $\ms D^\#$ of all homogenisations of elements of $\ms D$ is a subset of $\ms H$.\\
ii) Every strictly positive power $Y^n$ of an indeterminate $Y$ is in $\ms H$, and these are its only members when $k$ is algebraically closed. 
\end{rem}

\noindent We will use the algebraic and homogeneous signatures to characterise certain rings and fields naturally associated to EQAG over $k$. The first matter of business is to relate $\ms D$ and $\ms H$ through $\ms D^\#$, which amounts to characterising the elements of $\ms H$ in terms of $\ms D$, thanks to the folkloric

\begin{lem}\label{RELHOM0}
If $R$ is any ring and $P(\ov X,Y)$ is a homogeneous polynomial of degree $d$ with coefficients in $R$ and such that there exists in $P$ a pure monomial in $\ov X$, then for $Q(\ov X)=P(\ov X,1)$ we have $P=Q^\#(\ov X,Y)$.
\end{lem}

\begin{prop}\label{HOMSGN0}
A homogeneous polynomial $P(\ov X,Y)$ with coefficients in $k$ is in $\ms H$ if and only if it has the form $Y^e Q(\ov X,Y)$ with $Q\in \ms D^\#$, i.e. $P(\ov X,Y)=Y^e D^\#(\ov X,Y)$ with $D(\ov X)\in \ms D$.
\end{prop}
\begin{proof}
Let $Q(\ov X,Y)=D^\#(\ov X,Y)\in \ms D^\#$, $e\in\NN$ and $\ov a,b\in k$ : if $b^e Q(\ov a,b)=0$, then either $b^e=0$ or $Q(\ov a,b)=0$, and in either case we have $b=0$ (by Remark \ref{INCSIG0} for the second), so $Y^e Q(\ov X,Y)\in \ms H$. Conversely if $P(\ov X,Y)\in \ms H$, let $e$ be the greatest natural number $n$ such that $Y^n|P$, and write $P=Y^e Q(\ov X,Y)$. For $R(\ov X)=Q(\ov X,1)$ and $\ov a\in k$ such that $Q(\ov a,1)=R(\ov a)=0$, we have $P(\ov a,1)=0$ as well and by definition of $\ms H$ this means that $1=0$, which is impossible, so $R(\ov a)\neq 0$ and $R\in \ms D$. Now by choice of $e$, the homogeneous polynomial $Q$ contains a pure monomial in $\ov X$ and by Lemma \ref{RELHOM0}, we get $Q=R^\#(\ov X,Y)$ and the proof is complete.
\end{proof}

\begin{lem}\label{PERMSGN0}
The set $\ms D^\#$ is closed under variable permutations : if $P^\#(X_1,\ldots,X_n;X_{n+1})$ is in $\ms D^\#$, then for every permutation $\sigma\in \mf S_{n+1}$ we have $(P^\#)^\sigma\in \ms D^\#$. In particular, $\ms H$ itself is closed under permutation of variables. 
\end{lem}
\begin{proof}
Write $\ov X=X_1,\ldots,X_n$ and $P=\sum_{\ov i} a_{\ov i} \ov X^{\ov i}$, and let $d$ be the total degree of $P$ : with $Q=P^\#$, we have $Q=X_{n+1}^d P(\ov X/X_{n+1})$ in the fraction field of $k[X_i]$. Let $\sigma\in \mf S_{n+1}$. First suppose that $\sigma(n+1)=n+1$ : we have $Q^\sigma=\sum_{\ov i} a_{\ov i} \ov X^{\sigma^{-1} \ov i} X_{n+1}^{d-|\ov i|}=X_{n+1}^d P^\sigma(\ov X/X_{n+1})$; as $\ms D$ is closed under permutations of variables, we have $P^\sigma\in \ms D$ and we get $Q^\sigma=(P^\sigma)^\#$, whence $Q^\sigma\in \ms D^\#$. Now suppose that $\sigma(n+1)=j\neq n+1$, and let $\tau=(j,n+1)$ be the transposition exchanging $j$ and $n+1$ : we have $\tau\sigma(n+1)=n+1$, so by what precedes $Q^{\tau\sigma}\in \ms D^\#$, and it remains to prove that $Q^\tau\in \ms D^\#$. Write $Q^\tau=\sum_{\ov i} a_{\ov i} X_1^{i_1}\ldots X_j^{d-|\ov i|}\ldots X_n^{i_n} X_{n+1}^{i_j}$; we consider the polynomial $R(X_1,...,X_{j-1},X_{j+1},\ldots,X_{n+1})$ obtained by substituting $X_{n+1}$ for $X_j$ in $P$, i.e. $R=\sum_{\ov i} a_{\ov i} X_1^{i_1}\ldots X_{n+1}^{i_j}\ldots X_n^{i_n}$ : we have $R^\#(X_1,\ldots,X_{n+1};X_j)=\sum_{\ov i} a_{\ov i} X_1^{i_1}\ldots X_{n+1}^{i_j}\ldots X_n^{i_n} X_j^{d-|\ov j|}=Q^\tau(\ov X;X_{n+1})$, so we only need to prove that $R\in \ms D$, and this is obvious by definition of $R$, because $P\in \ms D$. We conclude that $Q^\sigma\in \ms D^\#$, which is then closed under permutation of variables. 
As for $\ms H$, if $P(\ov X,Y)\in \ms H$ has the form $Y^e D^\#(\ov X,Y)$ by Proposition \ref{HOMSGN0}, write $X_{n+1}=Y$, let $\sigma\in \mf S_{n+1}$ and write $\ov X'=\ov X X_{n+1}$ and $(\ov X')_\sigma=X_{\sigma (1)}\ldots X_{\sigma(n+1)}$ : we have $P^\sigma(\ov X,Y)=P((\ov X')_{\sigma})=X_{\sigma(n+1)}^e D^\#((\ov X')_\sigma)$; by the first part of the proof, we have $D^\#((\ov X')_\sigma)=(D^\#)^\sigma(\ov X')\in \ms D^\#$, so by Proposition \ref{HOMSGN0} again $P^\sigma\in \ms H$, and this completes the proof.
\end{proof}

\noindent In the light of the homogeneous signature, the characterisations of \ref{CARSPESIG} and \ref{CARSPESIG'} refer implicitly to all normic forms over $k$ :  

\begin{prop}\label{CHARSGNS}
i) The homogeneous signature $\ms H$ of $k$ is the set of all normic forms over $k$.\\
ii) The algebraic signature $\ms D$ of $k$ is the set of all deshomogenisations of normic forms over $k$.
\end{prop}
\begin{proof}
i) Every normic form over $k$ is obviously a member of $\ms H$, so let $P(\ov X,X_{n+1})\in\ms H$ with $ \ov X=X_1,\ldots,X_n$ and let $\ov a=a_1,\ldots,a_n$ and $b$ taken from $k$ and such that $P(\ov a,b)=0$. As $P\in\ms H$, we have $b=0$, so let $\tau\in\mf S_{n+1}$ be the permutation exchanging $i$ and $n+1$, for $i\in\{1,\ldots,n\}$ : we have $0=P(\ov a,b)=P^\tau(a_1,\ldots,a_{i-1},b,a_{i+1},\ldots,a_n,a_i)$, and as $P^\tau\in\ms H$ by Lemma \ref{PERMSGN0}, we get $a_i=0$. Therefore, we have $\ov a=\ov 0$ and $b=0$, so that $P$ is indeed a normic form over $k$.\\
ii) Let $D(\ov X)\in \ms D$ : by Remark \ref{INCSIG0}, we have $D^\#(\ov X,Y)\in\ms H$, so by (i), $D(\ov X)=D^\#(\ov X,1)$ is the deshomogenisation of a normic form over $k$. Conversely, if $P(\ov X,Y)$ is a normic form over $k$, then it is a member of $\ms H$ by (i) and thus has the form $Y^eD^\#(\ov X,Y)$ for $D(\ov X)\in \ms D$ by Proposition \ref{HOMSGN0} : it follows that $P(\ov X,1)=D(\ov X)\in\ms D$.
\end{proof}
\begin{rem}
Both signatures have a natural geometric interpretation : the algebraic signature is related to \emph{affine} EQAG in that it lists all the polynomials with no zero rational in the corresponding affine space, whereas the homogeneous signature is related to \emph{projective} EQAG in that it lists all the homogeneous polynomials with no zero rational in the corresponding projective space.
\end{rem}

\subsection{Special ideals and algebras}
The preliminary work on the homogeneous signature will clarify the discussion on special algebras, as abstracted from the characterisations of \ref{CARSPESIG} and \ref{CARSPESIG'}.

\begin{defi}\label{DEFSPE}
Let $A$ be a $k$-algebra and $I$ an ideal of $A$.\\
i) We say that $I$ is \emph{special} if for all $P(\ov X,Y)\in \ms H$ and $\ov a,b\in A$ such that $P(\ov a,b)\in I$, we have $b\in I$.\\
ii) We say that a $A$ is \emph{special} if the ideal $(0)$ in $A$ is special, i.e. if for all $P(\ov X,Y)\in\ms H$ and $\ov a,b\in A$ such that $P(\ov a,b)=0$, we have $b=0$. 
\end{defi}
\begin{rem}
i) By Lemma \ref{CARSPESIG} and Proposition \ref{HOMSGN0}, every essential ideal is special, but we will see that the converse is not true in general (e.g. Eample \ref{SPEPOLRAT}).\\
ii) In general, every special ideal is radical (by Remark \ref{INCSIG0}(ii)). If $k$ is algebraically closed, then $I$ is special if and only if it is radical, so $A$ is special if and only if $A$ is reduced.\\
iii) If $k$ is a real closed field (\cite{RAG}, Definition 1.2.1 or Section \ref{EXESPE}), then $I$ special if and only if it is real (\cite{RAG}, Definition 4.1.3). 
\end{rem}

\begin{prop}\label{CARPRIMSPE}
If $A$ is a $k$-algebra and $\p$ a prime ideal of $A$, then $\p$ is special if and only if for all $D(\ov X)\in\ms D$ and $\ov a,b\in A$ such that $D^\#(\ov a,b)\in\p$, we have $b\in\p$. In particular, an integral $k$-algebra $A$ is special if and only if for all $D(\ov X)\in\ms D$ and $\ov a,b\in A$ such that $D^\#(\ov a,b)=0$, we have $b=0$.
\end{prop}
\begin{proof}
If $\p$ is special and $D^\#(\ov a,b)\in\p$, by Remark \ref{INCSIG0} we have $D^\#\in\ms H$, so $b\in\p$.
Conversely, if the property holds and $P(\ov a,b)\in \p$ with $P(\ov X,Y)\in\ms H$ and $\ov a,b\in A$, then by Proposition \ref{HOMSGN0} $P$ has the form $Y^eD^\#(\ov X,Y)$, so that either $b^e\in \p$ or $D^\#(\ov a,b)\in\p$, and in both cases we get $b\in\p$.
\end{proof}

\begin{cor}\label{LOCSPEPRIM}
The localisation of a $k$-algebra $A$ at a special prime $\p$ of $A$ is a $*$-algebra.
\end{cor}
\begin{proof}
Let $D(\ov X)\in\ms D$ and $\ov a/s$ in $A_\p$ : as $s\notin \p$, we have $D^\#(\ov a,s)\notin \p$ by Proposition \ref{CARPRIMSPE}, so $D(\ov a/s)=(1/s^d).D^\#(\ov a,s)\in A_\p^\xx$. By definition, $A_\p$ is therefore a $*$-algebra.
\end{proof}

\noindent As regards field extensions of $k$, the special ones are exactly the $*$-algebras, and they are characterised by the \emph{algebraic} signature of $k$.

\begin{prop}\label{SPECEXT}
A field extension $K$ of $k$ is special if and only if $K$ preserves the \emph{algebraic} signature of $k$, if and only if $K$ is a $*$-algebra.
\end{prop}
\begin{proof}
If $A$ is any non-trivial special $k$-algebra and $D(\ov X)\in \ms D$, $\ov a\in A$, then as $1\neq 0$ and $D^\#\in\ms H$ we get $D(\ov a)=D^\#(\ov a,1)\neq 0$. Conversely, if the field extension $K$ preserves the algebraic signature of $k$ and $\ov a,b\in K$, with $b\neq 0$, we have $D^\#(\ov a,b)=b^d.D^\#(\ov a/b,1)=b^d.D(\ov a/b)\neq 0$ by hypothesis. By contraposition, if $D^\#(\ov a,b)=0$ we get $b=0$, so $K$ is special by Proposition \ref{CARPRIMSPE}.
As for the second equivalence, $K$ preserve the algebraic signature of $k$ if and only if every element of the form $D(\ov a)$ of $K$ is invertible, i.e. if and only if $K$ is a $*$-algebra.
\end{proof}

\begin{cor}\label{SPEFRAC}
An integral $k$-algebra $A$ is special if and only if its field of fractions $K=Fr(A)$ is.
\end{cor}
\begin{proof}
If $A$ is special, $D(\ov X)\in \ms D$ and $\ov a/s\in K$, we have $s\neq 0$ in $A$, so $D(\ov a/s)=(1/s^d).D^\#(\ov a,s)\neq 0$, and $K$ is special by \ref{SPECEXT}.
Conversely, any sub-algebra of a special algebra is obviously special.
\end{proof}

\noindent All this allows us to identify many natural special algebras associated to EQAG over $k$.

\begin{ex}\label{SPEPOLRAT}
If $V\subs k^n$ is an irreducible affine subvariety, then $k[V]$, $k\{V\}$ and $k(V)$ are special. In particular, the polynomial algebras $k[X_1,\ldots,X_n]$ and rational function fields $k(X_1,\ldots,X_n)$ are special.
\end{ex}
\begin{proof}
By definition, we have $k[V]=k[X_1,\ldots,X_n]/\ms I(\ms Z(I))$, and $\ms I(\ms Z(I))=\erad I$ is essential, so special by Lemma \ref{CARSPESIG} and Proposition \ref{HOMSGN0}. It follows that $k[V]$ is special, as well as $k(V)$ by Corollary \ref{SPEFRAC}, and $k\{V\}$ as a sub-algebra of $k(V)$.
\end{proof}

\noindent If the \og good" field extensions of $k$ are the special ones, as regards EQAG there are no rational points to look for in algebraic extensions of $k$ :

\begin{cor}\label{SPEXTRI}
Any special algebraic extension of $k$ is trivial.
\end{cor}
\begin{proof}
It suffices to prove it for any special monogenous finite extension $K=k[X]/(P)$, with $P$ irreducible. By Proposition \ref{SPECEXT} we have $(P)\cap M_A=\emptyset$ in $A=k[X]$, hence by Lemma \ref{CARSPE}, $(P)$ is an essential maximal ideal. By the \"Aquinullstellensatz \ref{NSREL}, $(P)$ has a zero rational in $k$, so $k\cong K$.
\end{proof}
\begin{rem}
This is an equiresidual analogue of the strong form of Hilbert's Nullstellensatz. 
\end{rem}

\subsection{Two natural examples}\label{EXESPE}
We illustrate some of the present concepts in two examples taken from number theory and pertaining to model theory : the field $\RR$ (and real closed fields), the fields $\QQ_p$ (and $p$-adically closed fields). 

\begin{lem}\label{EXSPRIM}
If $A$ is a $k$-algebra and $\p$ is an ideal disjoint from $M_A$ and maximal as such, then $\p$ is a special prime ideal.
\end{lem}
\begin{proof}
By maximality, $\p A_M$ is maximal in $A_M$, so $\p$ is prime, and $K:=A_M/\p A_M$ is a field and a $*$-algebra as a quotient of $A_M$, so by Proposition \ref{SPECEXT}, $K$ is special, and therefore $A/\p$ is special as a sub-algebra, whence $\p$ itself is special.
\end{proof}
\begin{rem}
It follows that there exists a morphism from a $k$-algebra $A$ into a special extension of $k$ if and only if $A_M$ is not trivial (compare with Lemma \ref{CARPRE}).
\end{rem}

\begin{lem}\label{SGNMOR}
If $A$ is a $k$-algebra, then $A$ preserves the algebraic signature of $k$ if and only if there exists a morphism from $A$ into a special extension of $K$.
\end{lem}
\begin{proof}
Let $\phi:A\to K$ be a $k$-morphism into a special extension of $k$ : if $D\in\ms D$ and $\ov a\in A$, then $D(\phi(\ov a))\neq 0$ by Proposition \ref{SPECEXT}, so $D(\ov a)\neq 0$ as well.
Conversely, if $A$ preserves the algebraic signature of $k$, then $0\notin M_A$, so by Zorn's Lemma let $\p$ be a special prime ideal of $A$ by Lemma \ref{EXSPRIM} : $K=Fr(A/\p)$ is a special extension of $k$ by Corollary \ref{SPEFRAC}.
\end{proof}

\noindent Recall that a field $R$ is called \emph{real} if $-1$ is not a sum of squares in $R$ (\cite{RAG}, 1.1.9), and \emph{real closed} if it is real and has no proper algebraic real extension (\cite{RAG}, 1.2.1). By \cite{RAG} 1.2.2, $\RR$ is the prototype of real closed fields. Now in a real field $R$, the existence of an ordering implies that all the polynomials $X_1^2+\ldots X_n^2$, $n\in\NN^*$, are normic forms.

\begin{prop}\label{REALSPEC}
If $R$ is a real closed field, then an $R$-algebra $A$ has a morphism into a special extension of $R$ if and only if $-1$ is not a sum of squares in $A$. In particular, the special field extensions of $R$ are its real extensions.
\end{prop}
\begin{proof}
As $R$ has a unique ordering in which the squares are the non-negative elements (\cite{RAG}, Theorem 1.2.2), the polynomials of the form $1+X_1^2+\ldots +X_n^2$ are in its algebraic signature. If $\phi:A\to R'$ is a morphism into a special extension of $R$, by Lemma \ref{SGNMOR} $-1$ is not a sum of squares in $A$.
Conversely, if $-1$ is not a sum of squares, then by \cite{RAG}, Theorem 4.3.7, there exists an $R$-morphism $\phi:A\to R'$ into a real closed field $R'$; by \cite{RAG}, Proposition 4.1.1, $R'$ preserves the algebraic signature of $R$, hence is a special extension by Proposition \ref{SPEFRAC}.
Finally, by what precedes, any special field extension of $R$ is real, and if $R'$ is a real field extension, then it embeds into a special extension $R''$, so $R'$ itself is special.
\end{proof}

\begin{rem}
In general EQAG may thus encompass already known algebraico-geometric theories over non-algebraically closed fields : here the EQAG of any real closed field is precisely its \emph{real} algebraic geometry, which is not obvious at first in that this last is developped in the category of \emph{all} real closed fields.
\end{rem}

\noindent In the very same spirit, we recall that if $p$ is a prime natural number, a field $F$ is analogously called \emph{formally $p$-adic} (over $\QQ$) if $1/p\notin \ZZ_{p\ZZ}[\gamma(F)]$, where for every $x\in F$, $\gamma(x)=(1/2p)(\frac{1}{x^p-x+1}+\frac{1}{x^p-x-1})$ and $\gamma(F)=\{\gamma(x) : x\in F\}$. A field $F$ is then called \emph{$p$-adically closed} if it is is formally $p$-adic and has no proper algebraic extension which is formally $p$-adic (\cite{SKP}, Definition 3). By \cite{SKP}, Proposition 4, $\QQ_p$ is the prototype of $p$-adically closed fields. In a formally $p$-adic field, the existence of a $p$-adic valuation entails that for all $n\in\NN^*$, the polynomial $X_1^n+pX_2^n+\ldots +p^{n-1} X_n^n=\sum_{i=1}^n p^{i-1} X_i^n$ is a normic form (\cite{SKP}, Theorem 4).\\
\noindent Now the condition defining a formally $p$-adic field $F$ may be restated as follows, since the ring $\ZZ_{p\ZZ}[\gamma(F)]$ is the set of all elements $P(\gamma(x_1),\ldots,\gamma(x_n))$ with $P\in \ZZ_{p\ZZ}[X_1,\ldots,X_n]$ and $x_1,\ldots,x_n\in F$. Rewriting and reducing to the same denominator, to each $P=\sum_{\ov i} a_{\ov i} \ov X^{\ov i}$ with total degree $d$ and maximal multidegree $\ov j=\max\{\ov i : a_{\ov i}\neq 0\}$ (for the product ordering on $\NN^n$) we associate the polynomial $\wt P=p^{d+1}\prod_{k=1}[(X_k^p-X_k)^2-1]^{j_k}\sum_{\ov i} a_{\ov i} p^{d-|\ov i|} \prod_{k=1}^n (X_k^p-X_k)^{i_k}$. A field $F$ is thus formally $p$-adic if and only if for all $n\in\NN^*$ and $P\in \ZZ_{p\ZZ}[X_1,\ldots,X_n]$, $\wt P -1$ has no zero rational in $F$.

\begin{prop}
If $F$ is a $p$-adically closed field, then an $F$-algebra $A$ has a morphism into a special extension of $F$ if and only if no $\wt P-1$ has a zero rational in $A$. In particular, the special field extensions of $F$ are its formally $p$-adic extensions.
\end{prop}
\begin{proof}
As $F$ is formally $p$-adic, by what precedes the polynomials of the form $\wt P-1$ are in its algebraic signature. If $\phi:A\to F'$ is a morphism into a special extension of $R$, by Lemma \ref{SGNMOR} again no such polynomial has a zero rational in $A$.
Conversely, if no $\wt P-1$ has a zero rational in $A$, then by Lemma \ref{EXSPRIM} let $\p$ be a special prime of $A$ disjoint from $M_A$, and $F'=Fr(A/\p)$ : no $\wt P-1$ has a zero rational in $F'$, so $F'$ is a formally $p$-adic field, which embeds into a $p$-adically closed field by \cite{SKP}, Corollary to Proposition 4. Now as $F'$ is $p$-adically closed, by \cite{MI1}, Theorem 1, it preserves the algebraic signature of $F$, and thus is a special extension.
We conclude as in \ref{REALSPEC}.
\end{proof}

\begin{rem}
Both examples clarify the signification of the homogeneous signature and the special algebras associated to any field $k$. Here we made essential use of the fact that the first order theories (of real closed fields and of $p$-adically closed fields) are \emph{model-complete} but in general, there is no reason why $k$ should be the model of such a theory. This is where \emph{positive logic} will powerfully provide in \cite{PAG} a complementary point of view as a means to unifying the present context and algebraic geometry relativised to a \emph{model complete} theory of fields (as naturally appearing in number theory).
\end{rem}

\section{Localisation and function rings}\label{LOCFUN}
The notion of \emph{special radical} of an ideal, in some sense the \og true" generalisation of the classic radical to EQAG, allows us to deal with localisation of special algebras, and then to describe the rings of rational functions on any affine variety as certain special $*$-algebras essentially of finite type. Localisation is also interpreted in this context in an analogue of the prime spectrum as a representation for special $*$-algebras, which connects EQAG to scheme theory.

\subsection{The special radical}
Essential maximal ideals of a $k$-algebra $A$ represent its rational points in $k$ and not every interesting, i.e. special, $k$-algebra, has such : for instance, if $V$ is an irreducible affine variety its field of fractions $k(V)$ is special, but has no morphism into $k$. If we rather consider rational points in special extensions of $k$, by Corollary \ref{SPEFRAC} we have to replace essential maximal ideals by special primes, and thus define another kind of radical :

\begin{defi}
If $A$ is a $k$-algebra and $I$ an ideal of $A$, we will say that the intersection of all special prime ideals containing $I$ is the \emph{special radical of $I$}, which we note $\srad I$. We always have $\srad I\subs \erad I$.
\end{defi}
\begin{rem}
The special radical 
coincides with the equiradical in finitely generated algebras and $*$-algebras of finite $*$-type (Corollary \ref{ESSEQUISPE}).
\end{rem}

\noindent In this extended context, the generalisation of Lemma \ref{CARPRE} is the following 

\begin{lem}\label{CARPRESPE}
If $A$ is a $*$-algebra, then every maximal ideal of $A$ is special. In particular, if $A\neq 0$ there exists a $k$-morphism of $A$ into a special extension of $k$.
\end{lem}
\begin{proof}
The category of $*$-algebras is stable under quotients, so if $\m$ is a maximal ideal of $A$, $A/\m$ is a field and a $*$-algebra, and thus a special extension of $k$ by Proposition \ref{SPECEXT}.
\end{proof}

\noindent With the same notations as in Section \ref{ERAD}, we have the following characterisation of the special radical :

\begin{thm}\label{CARSPERAD}
For any $k$-algebra $A$ and ideal $I$ of $A$, we have :\\
i) $\srad I=\{a\in A : I\cap \Sigma_a\neq\emptyset\}$\\ 
ii) $\srad I=\{b\in A : \exists m,n\in \NN,\exists D(\ov x)\in \ms D,\exists \ov a\in A,\ b^mD(\ov a/b^n)\in I_b\}$\\
iii) $\srad I=\{b\in A : \exists m,n\in\NN,\exists D(\ov x)\in \ms D, \exists \ov a\in A,\ b^mD^\#(\ov a,b^n)\in I\}$\\
iv) $\srad I=\{b\in A: \exists n\in \NN,\exists D(\ov x)\in \ms D,\exists \ov a\in A,b^nD^\#(\ov a,b^n)\in I\}$\\
v) $\srad I=\{b\in A : \exists P(\ov x,y)\in \ms H,\exists \ov a\in A,\exists n\in\NN,P(\ov a,b^n)\in I\}$.
\end{thm}
\begin{proof}
i) Suppose $a\notin \srad I$ : by definition there exists a special prime $\p$ of $A$ containing $I$ and such that $a\notin \p$. For all $m,n\in\NN$, we have $a^m,a^n\notin\p$ and thus for all $D\in \ms D$ and appropriate $\ov b\in A$ we have $a^mD^\#(\ov b,a^n)\notin \p$ by Proposition \ref{CARPRIMSPE}, so $I\cap \Sigma_a=\emptyset$ by primality of $\p$. 
Conversely, if $I\cap \Sigma_a=\emptyset$ then $A_{\<a\>}/\Sigma_a^{-1} I\neq 0$ and as $((A/I)_{a+I})_M\cong (A/I)_{\<a+I\>}$ (by Lemma \ref{CARUNLOC}) $\cong\Sigma_a^{-1}(A/I)\cong A_{\<a\>}/\Sigma_a^{-1} I$, by Lemma \ref{CARPRESPE} there exists a $k$-morphism from $(A/I)_{\<a+I\>}$ into a special extension $K$ of $k$, and the kernel $\p\sups I$ of the composite morphism $A\to A/I \to (A/I)_{\<a+I\>}\to K$ is a special prime with $a\notin\p$, so we get $a\notin \srad I$.\\ 
ii) Let $b\in\srad I$ and by (i) an element $b^m\prod_{i=1}^n D_i^\#(\ov a_i,b^{k_i})$ of $I\cap \Sigma_b$ : we have $I_b\ni b^m\prod_i b^{k_id_i} D_i(\ov a_i/b^{k_i})=b^m\prod_i b^{k_id_i} D_i(b^{k_i'} \ov a_i/b^k)$ (with $k=\max_i\{k_i\}$ and $k'_i=k+k_i$) $=b^{m'} D(\ov a'/b^k)$ in $A_b$, with $m',k\in\NN$. Conversely, if $b^mD(\ov a/b^n)\in I_b$, we may in $A_b$ write $b^mD(\ov a/b^n)=c/b^k$ with $c\in I$ and $k\in \NN$, i.e $b^mD^\#(\ov a,b^n)/b^{nd}=c/b^k$ : there exists $u\in\NN$ such that $b^u(b^{m+k}D^\#(\ov a,b^n)-cb^{nd})=0$ in $A$, and therefore $b^{u+m+k}D^\#(\ov a,b^n)=b^{u+nd}c\in I$, so that $b\in \srad I$ by (i) again.\\
iii) First, by (ii) we have $b\in \srad{(0)}\Iff\exists D(\ov x),\exists \ov a,\exists m,n\in\NN,\ b^m D(\ov a/b^n)=0\in A_b\Iff b^{m-nd} D^\#(\ov a,b^n)=0$ in $A_b\Iff b^m D^\#(\ov a,b^n)=0$ in $A_b\Iff \exists s\in \NN,b^{m+s} D^\#(\ov a,b^n)=0$ in $A$, so $b\in \srad{(0)}\Iff \exists D(\ov x),\exists \ov a,\exists m,n,\ b^m D^\#(\ov a,b^n)=0$. Secondly, apply this to $A/I$ : we have $b\in \srad I\Iff [b]\in \srad{(0)}$ in $A/I\Iff \exists D(\ov x),\exists\ov a\in A, \exists m,n\in\NN,[b^m D^\#(\ov a,b^n)]=[b]^m D^\#([\ov a],[b]^n)=0\Iff \exists D(\ov x),\exists \ov a,\exists m,n\in\NN, b^mD^\#(\ov a,b^n)\in I$.\\
iv) By (iii), if $b\in \srad I$ let $b^mD^\#(\ov a,b^n)\in I$ : by homogeneity of $D^\#$, we have \newline$b^{n+m} D^\#(b^m \ov a,b^{n+m})$ $=b^{n+m(d+1)}D^\#(\ov a,b^n)\in I$, each coordinate being multiplied by $b^m$. The other way is obvious.\\
v) Let $b\in \srad I$ : by (iv) there exist $D(\ov x),n$ and $\ov a$ with $b^n D^\#(\ov a,b^n)\in I$; as $P(\ov x,y):=yD^\#(\ov x,y)$ is in $\ms H$ by Proposition \ref{HOMSGN0}, the direct inclusion is obvious. Now if $P\in \ms H$ and $P(\ov a,b^n)\in I$, $P$ has the form $y^eD^\#(\ov x,y)$ by the same proposition, so $b^{ne}D^\#(\ov a,b^n)\in I$, and by (iii) we get $b\in \srad I$.
\end{proof}

\begin{cor}\label{CARSPEID}
If $A$ is a $k$-algebra, then an ideal $I$ of $A$ is special if and only if $I=\srad I$.
\end{cor}
\begin{proof}
Suppose $I$ is special and $a\in \srad I$ : by Theorem \ref{CARSPERAD}(iii) there exist $m,n,D(\ov x),\ov b$ such that $a^mD^\#(\ov b,a^n)\in I$, so $a^n\in I$ by hypothesis. As $I$ is radical, we have $a\in I$, hence $I=\srad I$.
Conversely, if $\srad I=I$ let $P(\ov x,y)\in \ms H$ and $\ov a,b$ such that $P(\ov a,b)\in I$ : by Theorem \ref{CARSPERAD}(v), we have $b\in I$, so $I$ is special. 
\end{proof}

\begin{cor}\label{ESSEQUISPE}
If $A$ is a finitely generated $k$-algebra or a $*$-algebra of finite $*$-type over $k$, then $A$ is special if and only if $A$ is essential.
\end{cor}
\begin{proof}
By Corollary \ref{CARSPEID}, $A$ is special if and only $(0)=\srad{(0)}$, if and only if for all $D(\ov x)\in\ms D$, $\ov a\in A$ and $m,n\in\NN$ such that $b^mD^\#(\ov a,b^n)=0$, we have $b=0$, by Theorem \ref{CARSPERAD}(iii). It follows that if $A$ is finitely generated, it is essential if and only if it is special by Lemma \ref{CARSPESIG}, and if $A$ is a $*$-algebra of finite $*$-type, it is essential if and only if it is special by Corollary \ref{CARSPESIG'}.
\end{proof}

\subsection{Localisation of special algebras and rings of rational functions}
Our first use of the special radical is associated to various localisations of $k$-algebras, and the stability of the special character under such. This naturally leads us to some function rings associated to EQAG.

\begin{lem}\label{INCLOCSRAD}
If $A$ is a $k$-algebra, $I$ is an ideal of $A$ and $S$ a multiplicative subset of $A$, then $(\srad{I})_S= \srad{I_S}$.
\end{lem}
\begin{proof}
First we prove that $(\srad I)_S\subs \srad{I_S}$ : the bijection $\p\mapsto \p_S$ exchanging the primes of $A$ disjoint from $S$ and the primes of $A_S$ exchanges the special primes disjoint from $S$ and the special primes of $A_S$. Indeed, if $\p$ is a special prime and $\p\cap S=\emptyset$, let $P(\ov x,y)\in\ms H$ and $\ov a/s,b/s\in A_S$ with $(1/s^d).P(\ov a,b)=P(\ov a/s,b/s)\in\p_S$ : we have $P(\ov a,b)\in\p$, so $b\in\p$ by Proposition \ref{CARPRIMSPE}, whence $b/s\in \p_S$, which is special. Conversely, if $\p_S\in Spec(A_S)$ is special and $P(\ov a,b)\in \p$, we have $P(\ov a/1,b/1)\in\p_S$, and thus $b/1\in \p_S$ by \ref{CARPRIMSPE} again, whence $b\in\p$, which is special. 
Now let $a/s\in (\srad{I})_S$ and $\p_S$ a special prime of $A_S$ containing $I_S$ : as $I\subs \p$ and $a/s=b/t$ for some $b\in \srad I$, there exists $u\in S$ such that $u(at-bs)=0$ in $A$ and as $\p\cap S=\emptyset$ and $b\in\p$ we get $a\in\p$, whence $a/s\in \p_S$, and thus $a/s\in \srad{I_S}$.
Secondly, we prove that $\srad{I_S}\subs (\srad I)_S$ : if $a/s\in \srad{I_S}$, by Theorem \ref{CARSPERAD}(v) there exist $P(\ov x,y)\in\ms H$, $n\in\NN$ and $\ov b/t$ such that $P(\ov b/t,(a/s)^n)\in I_S$; if $n=0$, then $(1/t)^d.P(\ov b,t)\in I_S$ and there exists $u\in S$ with $u.P(\ov b,t)\in I$, whence $I\ni (ua)^d.P(\ov b,t)=P(\ov bua,tua)$ (as $d\geq 1$), and therefore $atu\in \srad I$ by \ref{CARSPERAD}(v) again. If $n\geq 1$, so that also $(1/ts^n)^d.P(\ov bs^n,a^nt)\in I_S$, there exists $u\in S$ such that $u.P(\ov bs^n,a^nt)\in I$, and therefore $I\ni u^d.(tu)^{d(n-1)}.P(\ov bs^n,a^nt)=P(\ov bs^nt^{n-1}u^n,a^nt^nu^n)$, whence $atu\in \srad I$ by \ref{CARSPERAD}(v) again. In both cases, we get $a/s=atu/stu\in (\srad I)_S$, so that $\srad{I_S}=(\srad I)_S$ and the proof is complete. 
\end{proof}

\begin{cor}\label{LOCPRIMSPE}
If $A$ is a special $k$-algebra and $S$ a multiplicative subset of $A$, then $A_S$ is a special $k$-algebra as well. In particular, the canonical localisation of $A$ and the localisations of $A$ at special primes are special $*$-algebras.
\end{cor}
\begin{proof}
If $A$ is special, then $(0)=\srad{(0)}$ by Corollary \ref{CARSPEID}, whence $(0)_S=(\srad{(0)})_S=\srad{(0)_S}$ by Lemma \ref{INCLOCSRAD}, so $A_S$ is special as well.
Now if $\p$ is a special prime of $A$, then $A_\p$ is special by what precedes, and is a $*$-algebra by Corollary \ref{LOCSPEPRIM}.
\end{proof}

\begin{prop}\label{SPEFRAC2}
If $A$ is a special $k$-algebra, then the total ring $\Phi A$ of fractions of $A$ is a special $*$-algebra.
\end{prop}
\begin{proof}
As $A$ is special, then $\Phi A$ is special as well by Corollary \ref{LOCPRIMSPE}. Now let $D(\ov x)\in \ms D$ and $\ov a/s$ an appropriate tuple from $\Phi A$ : if $b\in A$ and $D^\#(\ov a,s)b=0$ in $A$, then either $D^\#$ is constant (and non-zero) and thus $b=0$, or $deg(D^\#)=d>0$ and $0=D^\#(\ov a,s)b^d=D^\#(b\ov a,bs)$ and as $A$ is special, we get $bs=0$, whence $b=0$ because $s$ is simplifiable; it follows that $D^\#(\ov a,s)$ is simplifiable, so $D(\ov a/s)=(1/s^d)D^\#(\ov a,s)\in (\Phi A)^\xx$, which is therefore a $*$-algebra.
\end{proof}
\begin{rem}
The proof that $\Phi A$ is special if $A$ is, is quite straightforward without the use of \ref{INCLOCSRAD}.
\end{rem}

\begin{cor}\label{SPEFRAC3}
Every special algebra of total fractions is a $*$-algebra.
\end{cor}

\noindent We are now in a position to incorporate rings of rational functions, over \emph{any} affine subvariety of $k$, into our constellation of algebraic objects.
If $R$ is any commutative ring, the product of any two non-zero simplifiable elements $a,b$ of $R$ is a non-zero simplifiable element, so we define a \emph{rational function on $V$} as an equivalence class $[(O,f/g)]$ of pairs $(O,f/g)$, $O$ being a \emph{dense} open subset of $V$ and $f,g\in k[V]$ with $g$ simplifiable and nowhere vanishing on $O$, where two pairs $(O,f/g)$ and $(U,l/h)$ are equivalent whenever $f/g$ and $l/h$ coincide on $O\cap U$. These classes naturally form a $k$-algebra $k(V)$, generalising the function field of an irreducible affine variety. 

\begin{defi}
The ring $k(V)$ is the \emph{algebra of rational functions over $V$}. We say that $\phi\in k(V)$ is \emph{regular at $P\in V$} if $\phi=[(O,f/g)]$ with $P\in O$, and the \emph{value of $\phi$ at $P$} is then $\phi(P)=f(P)/g(P)$.
\end{defi}

\begin{rem}\label{SIMPDENS}
Any open subset $O=D_V(g)$ of $V$ defined by a simplifiable element $g$ of $k[V]$ is dense in $V$.
\end{rem}
\begin{proof}
Let $P\in V$, and let $U=D_V(h)$ be a basic open subset of $V$ such that $P\in U$ :  we have $U\neq \emptyset$, so $h\neq 0$ and as $g$ is simplifiable, we get $gh\neq 0$, therefore $O\cap U=D_V(fg)\neq\emptyset$ : $O$ is dense in $V$.
\end{proof}

\noindent Now let $R_V$ be a total ring of fractions for $k[V]$ : associating to every $f/g\in R_V$ the class of $(D_V(g),f/g)$ in $A$, we define a $k$-morphism $\alpha:A\to k(V)$.

\begin{prop}\label{ISOFRAC}
The morphism $\alpha$ is an isomorphism.
\end{prop}
\begin{proof}
Suppose $\alpha(f/g)=0$ : the quotient $f/g$ defines the zero function on $D_V(g)$, which is dense in $V$ by Remark \ref{SIMPDENS}. In particular, $f$ vanishes on $D_V(g)$, and by continuity (as $\Delta_V=\{(x,x)\in k^n\xx k^n : x\in V\}$ is closed) we get $f=0$, therefore $\alpha$ is injective. Now if $[(O,f/g)]\in A$, by definition we have $D_V(g)\subs O$, so that $\phi(f/g)=[(O,f/g)]$ and $\alpha$ is an isomorphism.
\end{proof}

\begin{cor}
The ring $k(V)$ of rational functions over an affine subvariety $V$ is a special $*$-algebra.
\end{cor}

\subsection{The special prime spectrum of a special $*$-algebra}
If $A$ is a $k$-algebra, we define an \og equiresidual" analogue of the prime spectrum, the \emph{special prime spectrum} of $A$, as the set $Spec^2\ A$ of all \emph{special} prime ideals of $A$, topologised as usual. The space $X=Spec^2\ A$ is thus essentially the set of generic points of $A$ in special extensions of $k$, and we define a sheaf of rings $\O_X$ on $X$ as the restriction of the structure sheaf on $Spec\ A$ : if $U\subs X$ is open, we let $\O_X(U)$ be the set of all maps $s:U\to \coprod_U A_\p$, such that for each $\p\in U$, $\exists a,b\in A$, $\exists U'\subs U$ open, with $\p\in U'$ and $s|_{U'}\equiv a/b$. Thus we already know that for every $\p\in Spec^2\ A$ we have a natural isomorphism $\O_{X,\p}\cong A_\p$, given by $[f,U]\mapsto f(\p)$ for $\p\in U$ and $f\in \O_X(U)$. The first thing we notice is that :

\begin{rem}
$\O_X$ is a sheaf of $*$-algebras.
\end{rem}
\begin{proof}
If $D(\ov x)\in \ms D$ and $\ov s=s_1,\ldots,s_n\in \O_X(U)$ is an appropriate tuple of locals sections with $s_i\equiv a_i/b_i$ for all $i$, then for each $\p\in U$ we have $D(\ov s)(\p)=\delta /b^d$ in $A_\p$ with $\delta=D^\#(b'_1a_1,\ldots,b_n'a_n,b)$, $b'_i=\prod_{j\neq i} b_j$, $b=\prod_j b_j$ and $d=deg(D)$. As $D(\ov s)(\p)$ is invertible by Corollary \ref{LOCSPEPRIM}, we have $\delta\in A-\p$, and thus the element $t$ of $\O_X(U)$ defined by $b^d/\delta$ is an inverse for $D(\ov s)$.
\end{proof}

\noindent In order to complete the affine doctrine of EQAG, by analogy with $Spec$ we want to use $Spec^2$ as a geometric representation of certain $k$-algebras.

\begin{lem}\label{COMPSPEC}
For every $k$-algebra $A$, the space $X=Spec^2\ A$ is compact.
\end{lem}
\begin{proof}
Let $X=\bigcup_I D(a_i)$ be a basic open cover of $X$, so that $1\in \srad{\sum_I Aa_i}$ : by Theorem \ref{CARSPERAD}(iii) there exists $m\in M_A\cap \sum_I Aa_i$, hence a finite subset $I_0\subs I$ such that $m\in \sum_{I_0} Aa_i$. As no special prime $\p$ can contain $m$, we get $X=\bigcap_{I_0} D(a_i)$, and $X$ is compact.
\end{proof}

\begin{lem}\label{RADLOC}
Let $h\in A$, $\alpha\in \Sigma_h$ and $s:\p\in D(h)\mapsto g/h\alpha\in A_\p$. If $s\equiv 0$, then $gh\in \srad{(0)}$.
\end{lem}
\begin{proof}
Assume $g/h\alpha\equiv 0$ on $D(h)$ : for each $\p\in D(h)$, there exists $\mu\in A-\p$ such that $\mu g=0$, so $g\in\p$, whence $gh\in\p$, and if $\p\in V(h)$ we have $gh\in\p$ as well so that $gh\in \srad{(0)}$.
\end{proof}

\noindent For any $h\in A$,  we have a natural morphism $\psi:A_{\<h\>}\to \Gamma(D(h),\O_X)$ defined by $\psi(g/\alpha):= [\p\in D(h)\mapsto g/\alpha \in A_\p]$. By analogy with Theorem \ref{MAG.3.6.a}, we have the following characterisation of the ring of local sections $\Gamma(D(h),\O_X)$ :

\begin{thm}\label{SECSPEC2}
The $k$-morphism $\psi$ is surjective, and an isomorphism if $A_{\<h\>}$ is special.
\end{thm}
\begin{proof}
As for surjectivity, let $s\in \O_X(D(h))$ and $D(h)=\bigcup_I D(a_i)$ a basic open cover, with $s|_{D(a_i)}\equiv g_i/h_i$  for all $i$. As in \ref{MAG.3.6.a}, as $D(a_i)\subs D(h_i)$ we have $\srad{a_i}\subs \srad{e_i}$  , so by Theorem \ref{CARSPERAD} there exists $\alpha_i \in \Sigma_{a_i}$ and $g'_i\in A$ with $\alpha_i=g'_ih_i$, hence $s|_{D(a_i)}\equiv g_ig'_ia_i/a_i\alpha_i$ and we may assume that $D(a_i)=D(h_i)$ for all $i$.; by Lemma \ref{COMPSPEC}, we may write $D(h)=\bigcup_{I_0} D(h_i)$ with $I_0\subs I$ finite. For all $i,j\in I_0$, $g_i/h_i$ and $g_j/h_j$ coincide on $D(h_ih_j)=D(h_i)\cap D(h_j)$, so $h_ih_j(g_ih_j-g_jh_i)=0$ by Lemma \ref{RADLOC} again. Now as also $D(h)=\bigcup_{I_0} D(h_i^2)$, we have $h\in \srad{h_1^2,\ldots, h_m^2}$ and by \ref{CARSPERAD} again there exist $\alpha\in \Sigma_h$ and $a_i\in A$ such that $\alpha=\sum_{I_0} a_ih_i^2$. Let $\p\in D(h)$ : for each $j\in I_0$ such that $\p\in D(h_j)$, we have $h_j^2\sum_{I_0} a_ig_ih_i=\sum_{I_0} a_i g_jh_jh_i^2=g_jh_j\alpha$; as $s|_{D(h_j)}\equiv g_j/h_j$, we get $s.h_j^2\alpha\equiv g_jh_j\alpha\equiv h_j^2\sum_{I_0} a_ig_i h_i$ on $D(h_j)$, thus $s$ is represented by $\sum_{I_0} a_ig_ih_i/\alpha$ on each $D(h_j)$, and hence on $D(h)$, so $\psi$ is surjective.\\
As for injectivity, suppose $A_{\<h\>}$ is special and $\psi(g/\alpha)\equiv 0$ : the section of $\O_X$ defined on $D(h)$ by $gh/\alpha h$ is zero, hence $gh^2\in \srad{(0)}$ in $A$ by Lemma \ref{RADLOC}, whence $g/\alpha=gh^2/\alpha h^2\in \srad{(0)}$ in $A_{\<h\>}$ by Lemma \ref{INCLOCSRAD}, so $g/\alpha=0$ by hypothesis, and $\psi$ is injective.
\end{proof}

\noindent Theorem \ref{SECSPEC2} singles out again special $*$-algebras, which turn out to form a \og a good class" for which $Spec^2$ gives exact geometric representations :

\begin{cor}
If $A$ is a special $k$-algebra, then $\O_X$ is a sheaf of special $*$-algebras and $A_M$ is naturally isomorphic to the ring of global sections of $X=Spec^2\ A$. In particular, if $A$ is a special $*$-algebra then $A\cong \Gamma(X,\O_X)$.
\end{cor}
\begin{proof}
Suppose $A$ is special : by Corollary \ref{LOCPRIMSPE}, every localisation of $A$ is special, and by Theorem \ref{SECSPEC2}, for every basic open $U\subs X$ the ring $\O_X(U)$ is special; as special algebras are stable under projective limits, $\O_X$ is a sheaf of special $*$-algebras. By \ref{SECSPEC2} also, we have $\Gamma(X,\O_X)\cong A_{\<1\>}\cong A_M$, so if $A$ is furthermore a $*$-algebra, we have $A\cong A_M$, whence $A\cong \Gamma(X,\O_X)$.
\end{proof}

\noindent Finally, let $A$ be an affine $*$-algebra : the continuous mapping $\theta :\phi\in X=Sp_k A\mapsto Ker(\phi)\in Y=Spec^2\ A$ is clearly a homeomorphism onto its image, and a morphism of locally ringed spaces : if $s\in\O_Y(U)$ for $U\subs Y$ open and $s|_U\equiv a/b$, then $\theta_U^\#(s):\phi\mapsto \phi(a)/\phi(b)$ is in $\O_X(U)$, and the residual morphisms $\theta^\#_\phi:[s]\in \O_{Y,Ker(\phi)}\mapsto [\theta_U^\#(s)]\in \O_{X,\phi}$ are obviously local.
Now every maximal ideal of $A$ is essential, hence special, so the maximal ideals of $A$ are the closed points of $Y$ \emph{and} $Z=Spec\ A$. Furthermore, $Im(\theta)$ is \emph{dense} in $Z$ (and hence in $Y$) : if $\p\in D(a)$, a basic open of $Z$, as $a\neq 0$ there exists $\phi\in Sp_k\ A$ such that $\phi(a)\neq 0$ ($A$ being isomorphic to $K[V]_M$ for a $V\subs k^n$, its Jacobson radical is $(0)$) and thus $Ker(\phi)=\theta(\phi)\in D(a)$, so we have proved :

\begin{prop}
Every affine algebraic equivariety over $k$ is isomorphic to the dense subspace of closed points of an affine scheme over $k$.
\end{prop}

\section*{Conclusions}
So far we have built a robust theory of (equiresidual) affine algebraic varieties over any field and a promising extension of the usual commutative algebra underlying affine algebraic geometry over algebraically closed fields. We have laid the groundwork for a theory of algebraic equivarieties, which we will develop in \cite{EQAG2}, in which we expound the important particular case of quasi-projective equivarieties. 
From this point on, several directions may be pursued. First we wish to explore this subject further and investigate some usual constructions and theorems from classic algebraic geometry in the present setting, as the study of simple points and tangent spaces, of dimension global and local, of \'etale morphisms between algebraic equivarieties, and of differential forms, thanks to the formalism of canonical localisations and $*$-algebras, as well as normal varieties, if we may reinterpret the notion of integral dependence in the context of $*$-algebras. 
Another series of questions lies in the potential applications of EQAG to the \og inner" algebraic geometry of any particular field, using the concepts and tools presented here. A good start and key trial would be to sketch some general features of algebraic geometry over $\QQ$. 
In this perspective, normic forms have played a fundamental role in the present work, but were only used as a basic ingredient for the \"Aquinullstellensatz, the equiradical and the characterisation of the homogeneous signature and special algebras and ideals. They however appear in an axiomatisable setting in \cite{MK} in connexion with Galois extensions of the ground field, so this topic should be explored further in relation to Galois theory, hopefully connecting EQAG with algebraic geometry in the algebraic closure of a perfect ground field.
From another point of view, Section \ref{ESSPECDUAL} shows that the maximal spectrum functor is very well behaved with respect to all $*$-algebras of finite $*$-type whenever they are special (i.e essential). Along this line of thought, dropping the \og special" hypothesis there should be conceived an equiresidual alternative to schemes in order to take care of infinitesimals, maybe an analogue of the usual algebraic spaces. This would be closely related to the study of the interaction between EQAG and classic henselisation.

\noindent We have occasionally borrowed notions originally arising in first order logic, and transposed them here in the form of pure commutative algebra. We will develop the foundations for a universal framework for model-theoretic algebraic geometry in a subsequent work, the forthcoming \cite{PAG}, 
which deals with \og positive algebraic geometry", an interplay between EQAG, positive logic (a negation-free version of A. Robinson's theory of existential completeness) and quasivarieties (a \og subdoctrine" of first order logic).
With some background on \'etale morphisms of affine equivarieties, we wish to build on this foundation in order to investigate a ubiquitary type of theories of fields which appear in connexion to number theory (real closed fields, p-adically closed fields, complete theories of pseudo-algebraically closed fields, ...), and which have been recognised by McKenna in \cite{MK} and systematised by B\'elair in \cite{BEL} thanks to the work of Robinson in \cite{ER}, joining forces with the tradition of coherent logic and hopefully connecting with topos theory in order to provide a natural framework for algebraic geometry in a \og non-spatial" setting. The archetypical example of pseudo-algebraically closed fields will fall into this field of investigation and we wish the present work and some elements of this program to be of some use to \og Field Arithmetics" (see \cite{FA}), where one's particular interest lies in algebraic geometry over many fields which are not algebraically closed.

\bibliography{EQAG1_2022-06-03}
\bibliographystyle{alpha}

\vspace{0.5cm}
\textit{\footnotesize E-mail adress} : \texttt{\footnotesize jeb.math@gmail.com}

\end{document}